\documentclass[12pt]{amsart}
\usepackage{
  a4wide, 
  amsmath,
  amsfonts,
  amssymb,
  graphicx, 
  times,
  color,
  hyphenat,
  pinlabel,
  stmaryrd, mathabx} 
\usepackage{pgfopts, ytableau}
\input xy
\xyoption{all}
\newcommand{\g}{{\mathfrak{g}}}
\newcommand{\sll}{{\mathfrak{sl}_{n+1}}}
\newcommand{\sln}{{\mathfrak{sl}_{n}}}
\newcommand{\so}{{\mathfrak{so}_{2n}}}
\newcommand{\soo}{{\mathfrak{so}_{2n+1}}}
\newcommand{\spp}{{\mathfrak{sp}_{2n}}}

\newcommand{\und}[1]{{\underline{#1}}}

\newcommand{\llambda}{{\bar{\lambda}}}
\newcommand{\lmu}{{\bar{\mu}}}

\newcommand{\brak}[1]{\langle #1\rangle}
\newcommand{\n}{\noindent}

\DeclareMathOperator{\amod}{\mbox{-}mod}

\DeclareMathOperator{\prmod}{\mbox{-}pmod}

\DeclareMathOperator{\ext}{ext}

\DeclareMathOperator{\gdim}{gdim}

\DeclareMathOperator{\seq}{Seq}

\newcommand{\figins}[3] 
{\raisebox{#1pt}{\includegraphics[height=#2 in]{figs/#3}}}

\newtheorem{thm}{Theorem}[section]
\newtheorem{lem}[thm]{Lemma}
\newtheorem{cor}[thm]{Corollary}
\newtheorem{prop}[thm]{Proposition}

\newtheorem{conj}[thm]{Conjecture}

\theoremstyle{definition}
\newtheorem{defn}[thm]{Definition}

\newcommand{\bZ}{\mathbb{Z}}
\newcommand{\bQ}{\mathbb{Q}}

\newcommand{\cD}{\mathcal{D}}

\newcommand{\cS}{\mathcal{S}}

\DeclareMathOperator{\Hom}{Hom}

\DeclareMathOperator{\Ind}{Ind}
\DeclareMathOperator{\Res}{Res}
\DeclareMathOperator{\res}{res}

\newcommand{\xra}[1]{\xrightarrow{#1}}

\newcommand{\ket}[1]{|\,#1\,\rangle}

\catcode`\@=11
\long\def\@makecaption#1#2{%
    \vskip 10pt
    \setbox\@tempboxa\hbox{%
\small{#1: }\ignorespaces #2}%
    \ifdim \wd\@tempboxa >\captionwidth {%
        \rightskip=\@captionmargin\leftskip=\@captionmargin
        \unhbox\@tempboxa\par}%
      \else
        \hbox to\hsize{\hfil\box\@tempboxa\hfil}%
    \fi}
\newdimen\@captionmargin\@captionmargin=2\parindent
\newdimen\captionwidth\captionwidth=\hsize
\catcode`\@=12
\title{KLR algebras and the branching rule II: The categorical Gelfand-Tsetlin basis for the classical Lie algebras}
\author{Pedro Vaz}
\address{Institut de Recherche en Math\'ematique et Physique\\
Universit\'e Catholique de Louvain\\ 
Chemin du Cyclotron 2\\ 
1348 Louvain-la-Neuve\\ 
Belgium}
\email{pedro.vaz@uclouvain.be}
\begin{document}
%
%
\newdimen\captionwidth\captionwidth=\hsize
%
%
\begin{abstract}
We construct functors categorifying the 
branching rules for $U_q(\g)$ for
$\g$ of type $B_n$ $C_n$, and $D_n$ for 
the embeddings  
$\soo\supset\mathfrak{so}_{2n-1}$,
$\spp\supset\mathfrak{sp}_{2n-2}$, and
$\so\supset\mathfrak{so}_{2n-2}$. 
We give the correspon\-ding categorical Gelfand-Tsetlin basis. 
\end{abstract}
\maketitle

%
%
\pagestyle{myheadings}
\markboth{\em\small Pedro Vaz}{\em\small KLR algebras and the branching rule II}
%
%
%
\section{Introduction}\label{sec:intro}

In this sequel to~\cite{vaz2} we continue the program of studying 
the categorical branching rules 
and the correponding categorical Gelfand-Tsetlin basis for $U_q(\g)$ with $\g$ 
a classical Lie algebra. 
Recall that a \emph{branching rule} for (quantum) Kac-Moody algebras 
is a procedure to write irreducibles of $\g$ as direct sums of irreducibles of $\g'$  
whenever we have an embedding $\g'\hookrightarrow\g$. 

Let $\Gamma_{n-1}$ be a Dynkin diagram of classical type with $n-1$ nodes. 
There is only one way of adding a node and the corresponding edge to $\Gamma_{n-1}$ 
to obtain a new Dynkin diagram $\Gamma_{n}$ \emph{of the same type} if 
$\Gamma_{n-1}$ is of type $B_n$, $C_n$ or $D_n$, and two equivalent ways if it is of type $A_n$.
This defines embeddings of corresponding root systems and consequently 
of the respective quantum Kac-Moody algebras.   
In this paper we only consider the branching rule for this kind of embedding.

We denote by $\g_n$ the quantum Kac-Moody algebra of $\Gamma_n$.
Any such $\g_n$ can be constructed from the empty Dynkin diagram 
by adjusting nodes one at a time in such a way that respects its type. 
Repeated application of the branching rule one decompose an irreducible $\g_n$-module 
$M^\lambda$
as a direct sum of 1-dimensional vector spaces. 
The vector space spanned by the collections of all paths in the weight space in going from 
$M^\lambda$ to such a collection of 1-dimensional vector spaces gives a distinguished basis 
of $M^\lambda$ called the Gelfand-Tsetlin basis. 

\smallskip

The purpose of the program described in~\cite{vaz} (see also~\cite{VW}) is to study categorical versions 
of the above and its applications to topology. 
To this end we make use of the program of categorification of quantum quantum groups 
developed by Khovanov, Lauda and Rouquier in~\cite{KL1,Rouq1}. 

\smallskip

In~\cite{vaz2} the author obtained the categorical branching rule and the corresponding 
catego\-rical Gelfand-Tsetlin basis for $\sll\supset\sln$. 
In a nutshel the categorical branching rule takes the form of a certain functor 
between categories of (graded, finite dimensional) 
modules over certain cyclo\-tomic Khovanov-Lauda-Rouquier algebras  
commuting with the catego\-rical $\sln$-action. 

\smallskip

The reader should be aware that this paper relies heavily 
in its predecessor~\cite{vaz2} and we do assume acquaintance with it.  
Most of the arguments given in~\cite{vaz2} for type $A_n$ 
can be transposed without difficulty to the types treated here. 
To avoid this paper from becoming oversized we often just state the result and refer to the reasoning 
followed in~\cite{vaz2}.    

\smallskip

The outline of the paper is the following.
In Section~\ref{sec:basics} we recall some aspects of the quantum algebras we deal with 
together with its representations, branching rules and Gelfand-Tsetlin basis.
In Section~\ref{sec:KLRalgebras} we describe the Khovanonv-Lauda-Rouquier algebras and state some 
of its properties that will be used in this sequel.
In Section~\ref{sec:catBR} we give the categorical branching rule by constructing 
appropriate surjections between KLR algebras. 
In Section~\ref{sec:catGT} we use the results of the previous section to define a collection of modules and prove 
their image in the Grothendieck group coincides with the Gelfand-Tsetlin basis.

\section{Quantum groups, the branching rule and the Gelfand-Tsetlin basis}\label{sec:basics}

Let $\g$ be a Lie algebra with Dynkin diagram $\Gamma$.
We denote the weight lattice by $\Lambda(\g)$ and the root lattice by $X(\g)$. 
Let $\{\alpha_i\}_{i\in\Gamma}$ be the simple roots and 
$\{\alpha_1^\vee\}_{i\in\Gamma}$ the simple coroots. 
For $\llambda\in\Lambda(\g)$ we define  
$\llambda_i= \alpha_i^\vee(\llambda)$. 
Denote the set of dominant integral weights by 
\begin{equation*}
\Lambda_+(\g) = \bigl\{ \llambda\in\Lambda(\g)\vert \alpha_i^\vee(\llambda)\in\bZ_{\geq 0}\text{ for all }i\in\Gamma \bigr\}. 
\end{equation*}
Denote by $\brak{\ ,\,}$ is the usual (symmetrized) inner product on 
$\Lambda(\g)$
and write $q_i= q^{d_i}$ where
$d_i=\tfrac{1}{2}\brak{\alpha_i,\alpha_i}$. 
Let also $[a]_i = q_i^{a-1} + q_i^{a-3} + \dotsm + q_i^{-a+1}$.
Moreover, let 
$c_{ij}=\alpha_i^\vee(\alpha_j)$
be the entries of the cartan matrix of $\g$. 

\medskip

The {\em quantum universal enveloping linear algebra} $U_q(\g)$ of $\g$ 
is 
the associative unital $\bQ(q)$-algebra generated by the Chevalley generators $F_i$, $E_i$ and 
$K_\eta$, for $i\in\Gamma$ and $\eta\in\Lambda(\g)$, 
subject to the relations
\begin{gather*}
K_0 = 1, \mspace{40mu} K_\eta K_{\eta'} = K_{\eta + \eta'},\ \text{ for all } \eta,\eta'\in\Lambda(\g),
\\
K_\eta F_i = q^{-\alpha_i^\vee(\eta)}F_iK_\eta, \mspace{40mu} 
K_\eta E_i = q^{\alpha_i^\vee(\eta)}E_iK_\eta,\ \ \text{ for all } \eta\in\Lambda(\g),
\\
E_iF_j - F_jE_i = \delta_{i,j}\dfrac{K_{d_i\alpha_i} - K_{-d_i\alpha_i} }{q_i - q_i^{-1}}.
\intertext{Moreover, for all $i\neq j$,}
\sum\limits_{a+b=-c_{ij}+1}(-1)^aF_i^{(a)}F_jF_i^{(b)} = 0
\mspace{40mu}\text{and}\mspace{40mu}
\sum\limits_{a+b=-c_{ij}+1}(-1)^aE_i^{(a)}E_jE_i^{(b)} = 0
\end{gather*} 
where $E_i^{(a)}=\frac{E_i^a}{[a]_i!}$.  

\smallskip 

For $\und{i}=(i_1,\dotsc , i_k)$ we define $F_{\und{i}} = F_{i_1}\dotsm F_{i_k}$ 
and $E_{\und{i}} = E_{i_1}\dotsm E_{i_k}$. 

\smallskip

\n The \emph{lower half} $U_q^-(\g)\subset U_q(\g)$ quantum algebra is the subalgebra generated by the $F$s 
(analogously for the \emph{upper half} $U_q^+(\g)$).

\medskip

Recall that a subspace $V_\lmu$  of a 
finite dimensional $U_q(\g)$-module $V$ is called a \emph{weight space} 
of weight $\mu(\lambda)$ if
\begin{equation*}
K_\mu v = q^{\mu(\lambda)}v
\end{equation*}
for all $v\in V_\lmu$ and that $V$ is called a \emph{weight module} if
\begin{equation*}
V = \bigoplus\limits_{\lmu\in\Lambda(\g)}V_{\lmu} . 
\end{equation*}

A weight module $V$ is called a \emph{highest weight module with highest weight $\llambda$} if 
there exists a nonzero weight vector vector $v_\llambda\in V_\llambda$ such that $E_iv_\llambda=0$ for all $i\in\Lambda$. 
Recall that for each $\llambda\in\Lambda_+(\g)$ there exists a unique irreducible highest 
weight module with highest weight $\llambda$. 

\smallskip

Let $\phi$ be the $q$-antilinear involution on $U_q(\g)$ defined by
\begin{equation*}
\phi(K_{\eta})=K_{-\eta}
\mspace{20mu}
\phi(F_i)=q_i^{-1}K_{d_i\alpha_i}E_i
\mspace{20mu}
\phi(E_i)=q_i^{-1}K_{-d_i\alpha_i}F_i .
\end{equation*}
The \emph{$q$-Shapovalov form} $\langle\ \,,\ \rangle$ is the unique non-degenerate symmetric bilinear form 
on the highest weight module $V(\llambda)$ satisfying 
\begin{align*}
\langle v_\llambda,  v_\llambda \rangle &= 1,
\\
\langle uv,v'\rangle &=  \langle v,\phi(u)v'\rangle \text{\quad for all $u\in U_q(\g)$ and $v,v'\in V(\lambda)$},
\\
f\langle v,v'\rangle &=  \langle\bar{f}v,v'\rangle = \langle v,fv'\rangle 
\text{\quad for any $f\in\bQ(q)$ and $v,v'\in V(\lambda)$}.
\end{align*}

\subsection{Branching rules and the Gelfand-Tsetlin basis}\label{ssec:brulesGT}

Let $\g_m$ be a Lie algebra with $\vert\Gamma\vert=m$. 
We denote by $V^{\g_m}_{\lambda}$ the irreducible finite dimensional representation of $U_q(\g_m)$ of 
highest weight $\lambda$.  
For the embedding $\g_{n-1}\hookrightarrow\g_n$ corresponding to 
adding one vertex to the Dynkin diagram of $\g_{n-1}$ 
the branching rule from~\cite{molev} says that   
\begin{equation}
\label{eq:brulsalg}
\bigl(V^{\g_n}_{\lambda}\bigr)^{\g_{n-1}} 
\cong 
\bigoplus\limits_{\mu} c(\mu) V_{\mu}^{\g_{n-1}} 
\end{equation}
is an isomorphism of $\g_{n-1}$-modules, where 
the multiplicity $c(\mu)$ 
of $V_{\mu}^{\g_{n-1}}$ as a direct summand of $V_{\mu}^{\g_n}$  
is  well-determined 
for all $\g$ of classical type (see~\cite{molev}). 
To make easier dealing with the multiplicities we first interpret $c(\mu)$ as a 
the cardinality of a certain set $C(\mu)$ and then define $\tau(\lambda)=\tau_{\g}(\lambda)$ 
as the set of ordered pairs 
$(\nu,\mu )$ such that $\mu$ occurs as a wighest weight in the decomposition~\eqref{eq:brulsalg} 
and $\nu\in C(\mu)$.  
We can then re-write formula~\eqref{eq:brulsalg} as  
\begin{equation}
\label{eq:brulsalgg}
\bigl(V^{\g_n}_{\lambda}\bigr)^{\g_{n-1}} 
\cong 
\bigoplus\limits_{(\nu,\mu)\,\in\,\tau(\lambda)} V_{\mu}^{\g_{n-1}} 
.
\end{equation}

We can then
apply the branching rule recursively until 
we end up with a direct sum of 1-dimensio\-nal spaces 
corresponding to a final decomposition of each irreducible of $\mathfrak{sl}_2$ 
into 1-dimensional $\bQ(q)$-vector spaces.  

\medskip

Let $\mu^{(1)}$  be a $\g_n$-highest weight. 
We say a sequence $(\mu^{(1)},\nu^{(1)},\mu^{(2)},\nu^{(2)},\mu^{(3)}, \dotsc ,\xi^{(n)})$, 
where $\xi^{(n)}=\nu^{(n)}$ for $\g_n$ of type $B_n$ or $C_n$ or $\xi^{(n)}=\mu^{(n)}$ for 
$\g_n$ of type $D_n$, 
is a \emph{Gelfand-Tsetlin pattern for $\mu^{(1)}$} 
if for each $j$ we have $(\nu^{(j)},\mu^{(j+1)})\in\tau(\mu^{(j)})$. 
We denote by $\cS(\lambda)$ the set of all the Gelfand-Tsetlin patterns for a
$\g_n$-highest weight $\lambda$.  

\medskip

There is a 1:1 correspondence 
between $\cS(\lambda)$ and the set of all the 1-dimensional spaces that we obtain using recursive 
application of the branching rules.
Let $V_{\cS(\lambda)}$ be the $\bQ(q)$-linear space spanned by $\cS(\lambda)$. 
We write $\ket{s}$ for an Gelfand-Tsetlin pattern $s$ seen as an element of $V_{\cS(\lambda)}$. 
It turns out that $V_{\cS(\lambda)}$ is isomorphic to  $V_\lambda^{\g_n}$ not only as vector spaces by as $\g_n$-modules. 

The $\g_n$-action on $V_{\cS(\lambda)}$ can be obtained through a procedure which, in some sense, 
is the reverse of the direct sum decomposition~\eqref{eq:brulsalgg} using the branching rule.   
While the generators $\{E_i,F_i\}_{i\in\{2,\dotsc ,n\}}$ preserve the weight spaces 
$V_\mu^{\g_{n-1}}$
on the right-hand side of~\eqref{eq:brulsalgg},   
the generators $E_1$ and $F_1$ move between the different $V_\mu^{\g_{n-1}}$ 
(two copies of $V_\mu^{\g_{n-1}}$ for the same $\mu$ and different $\nu$ count as distinct): 
Let $\phi$ be the isomorphism $V_\lambda^{\g_n}\to\oplus_{(\nu,\mu)\in\tau(\lambda)} V_{\mu}^{\g_{n-1}}$ 
in~\eqref{eq:brulsalgg}. 
Then the $\g_{n-1}$-action on $\oplus_{\mu\in\tau(\lambda)} V_{\mu}^{\g_{n-1}}$ extends to an $\g_n$-action 
if we define 
\begin{equation*}
E_1v = \phi E_1\phi^{-1} v
\mspace{25mu}\text{and}\mspace{25mu}
F_1v = \phi F_1\phi^{-1} v
\end{equation*}
for $v\in\oplus_{(\nu,\mu)\in\tau(\lambda)} V_{\mu}^{\g_{n-1}}$. 
This is a consequence of $\phi E_i\phi^{-1} v = E_i v$ and $\phi F_i\phi^{-1} v = F_i v$ 
for all $v\in\oplus_{(\nu,\mu)\in\tau(\lambda)} V_{\mu}^{\g_{n-1}}$. 
We can continue and intertwine this procedure with the branching rule procedure until we get the desired 1-dimensional spaces 
and regard the $\g_n$-action on them as an action on $V_{\cS(\lambda)}$. 
The basis of $V_{\cS(\lambda)}$ given by the Gelfand-Tsetlin patterns (i.e. elements of $\cS(\lambda)$) 
is called the Gelfand-Tsetlin basis for $V_\lambda^{\g_n}$.  
The explicit form of action of the generators of the Lie algebra 
$\g_n$ on the Gelfand-Tsetlin basis can be found for example in~\cite{molev} or~\cite{zelobenko}.

\subsection{Representations of quantum $\soo$}\label{ssec:qalgB}
Let 
$\epsilon_i=(0,\ldots,1,\ldots,0)\in \bZ^n$, with $1$ being on the $i$-th 
coordinate, for 
$i=1,\ldots,n$. We also denote usual inner product on $\bZ^n$ by  
$(\epsilon_i,\epsilon_j)=\delta_{i,j}$. 

\smallskip

Recall that the quantum enveloping algebra (QEA) of type $B_n$ corresponds to the Dynkin diagram
\begin{equation*}
\labellist
\pinlabel $\dotsc$  at 188 60.5    
\tiny \hair 2pt 
\pinlabel $1$   at   7 36  
\pinlabel $2$   at  78 36 
\pinlabel $n-2$ at 282 35 
\pinlabel $n-1$ at 358 35 
\pinlabel $n$   at 433 35 
\endlabellist
\figins{-19}{0.65}{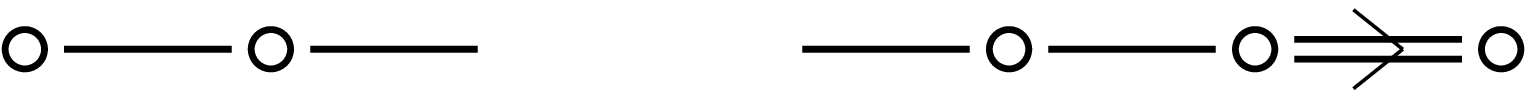} 
\end{equation*}
The simple root system is 
\begin{align*}
\alpha_1 = \varepsilon_1-\varepsilon_2,\ \dotsc ,\ \alpha_{n-1} = \varepsilon_{n-1}-\varepsilon_n,\ \alpha_n = \varepsilon_n . 
\end{align*}
Irreducible representations of $\soo$ correspond to n-tuples 
$\lambda_1, \dotsc ,\lambda_n$ which are all non-negative integers (for vector representations) 
or all non-negative half-integers (for spin representations) and they satisfy
\begin{equation*}
\lambda_1 \geq \lambda_2 \geq \dotsm \geq \lambda_{n-1} \geq \lambda_n \geq 0 .
\end{equation*}

The correspondence with the set of non-negative integers identifying $\lambda$ (Dynkin labels) is
\begin{equation*}
\begin{cases}
\widebar{\lambda}_j = \lambda_j - \lambda_{j+1} & \text{ for } j=1,\dotsc ,n-1, 
\\[2ex]
\widebar{\lambda}_n = 2 \lambda_{n}.
\end{cases}
\end{equation*}

\subsubsection{Branching rules}\label{ssec:brulesB}

Let $\Gamma_{n}$ and $\Gamma_{{n-1}}$ be the Dynkin diagrams associated to the types 
$B_n$ and $B_{n-1}$ 
with the vertices of $\Gamma_{n}$ labeled from 1 to $n$ and the vertices from $\Gamma_{{n-1}}$ 
labeled from 2 to $n$. 
Consider the inclusion $\Gamma_{{n-1}}\hookrightarrow\Gamma_{n}$ 
that adds a vertex labeled 1 at the beginning of $\Gamma_{{n-1}}$ and the corresponding edge, \emph{i.e.}   
\begin{equation}\label{eq:dynkinplusB}
\labellist
\pinlabel $\dotsc$  at 118 60.5    
\tiny \hair 2pt 
\pinlabel $2$   at   7 36  
\pinlabel $n-2$ at 215 35 
\pinlabel $n-1$ at 288 35 
\pinlabel $n$   at 364 35 
\endlabellist
\figins{-19}{0.65}{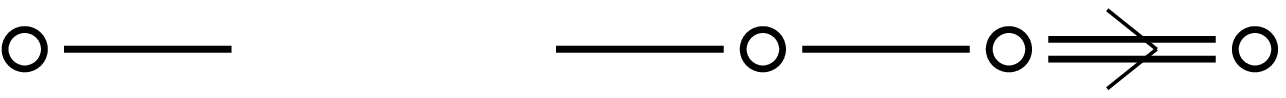}
\mspace{35mu}\hookrightarrow\mspace{35mu}  
\labellist
\pinlabel $\dotsc$  at 188 60.5    
\tiny \hair 2pt 
\pinlabel $1$   at   7 36  
\pinlabel $2$   at  78 36 
\pinlabel $n-2$ at 287 35 
\pinlabel $n-1$ at 364 35 
\pinlabel $n$   at 435 35 
\endlabellist
\figins{-19}{0.65}{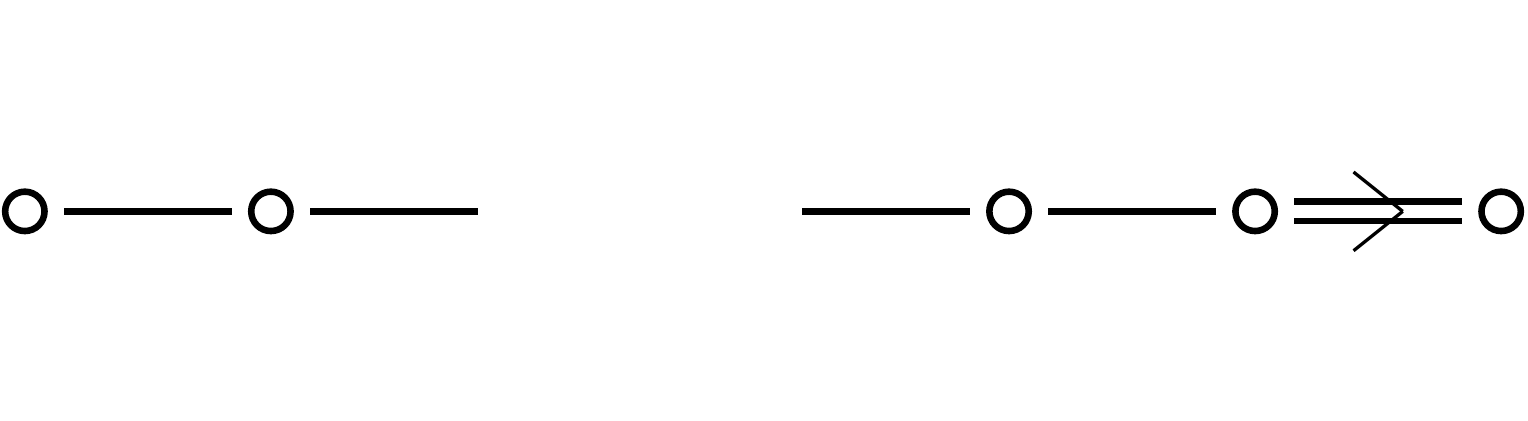} 
\end{equation}

We have isomorphisms of quantum $\mathfrak{so}_{2n-1}$-representations,
\begin{equation}\label{eq:brBB}
\bigl(V_\lambda^{\soo}\bigr)^{\mathfrak{so}_{2n-1}} \cong \bigoplus\limits_{\mu} c(\mu) V_\mu^{\mathfrak{so}_{2n-1}} ,
\end{equation}
where $\mu=(\mu_2,\dotsc ,\mu_n)$ is in $\Lambda_+^{\mathfrak{so}_{2n-1}}$
and
the multiplicity $c(\mu)$ is the number of 
$\nu = (\nu_1,\dotsc ,\nu_n)$ (all simultaneously integers or half-integers together with the $\lambda_i$s) satisfying 
\begin{align*}
\lambda_1 \geq \nu_1 \geq \lambda_2 \geq \dotsm \geq \nu_{n-1} \geq \lambda_n \geq \nu_n \geq -\lambda_n
\\
\nu_1 \geq \mu_2 \geq \nu_2\geq \dotsm \geq \nu_{n-1}\geq \mu_{n} \geq \nu_n \geq -\mu_n
\end{align*}
This can be depicted schematically as the \emph{Gelfand-Tsetlin pattern} 
(GT-pattern)
of type $B_n$
\begin{equation}\label{eq:GTpattB}
\left\vert
\begin{array}{ccccccccccc}
\lambda_1 & & \lambda _2 & & \lambda_3 & & \dotsm & & \lambda_n & & -\lambda_n
\\
& \nu_1 & & \nu_2 & & \dotsm & & & & \nu_n &
\\
 & & \mu _2 & & \mu_3 & & \dotsm & & \mu_n & & -\mu_n
\end{array}
\right.\xy/r0.62pt/:
(0,37)*{}; (12,0)*{};  **\dir{-};(0,37)*{}; (12,0)*{};  **\dir{-};(0,37)*{}; (12,0)*{};  **\dir{-};
(12,0)*{}; (0,-37); **\dir{-}; (12,0)*{}; (0,-37); **\dir{-}; (12,0)*{}; (0,-37); **\dir{-}
\endxy  \ ,
\end{equation}
where, for each triple of symbols $a_{i-1}$, $a_i$ and $a_{i+1}$ placed respectively in column 
$i-1$, $i$ and $i+1$, we have the inbetweenness condition
\begin{equation}\label{eq:inbetB}
a_{i-1} \geq a_i \geq a_{i+1} .
\end{equation}
The set $\tau_{B_n}(\lambda)$ denotes the set of all such $(\nu,\mu)$.

Notice that $\nu_n$ can be negative and that $\nu$ satisfies 
\begin{equation*}
\nu_1\geq\nu_2\geq \dotsc \geq\nu_{n-1}\geq \vert \nu_n\vert .
\end{equation*}

Let $\xi_{-i}\colon \bZ^n\to \bZ^n$  and  $\xi_{+i}\colon \bZ^n\to \bZ^{n-1}$ be the lattice maps 
defined by 
\begin{align*}
\xi_{-i}(\lambda) &= (\lambda_1, \dotsc ,\lambda_{i}-1,\dotsc,\lambda_n) 
\\
\xi_{+i}(\lambda) &= (\lambda_2, \dotsc ,\lambda_{i+1}+1,\dotsc,\lambda_n) 
\end{align*}
where $-\lambda_n\leq \bigl(\xi_{-n}(\lambda)\bigr)_n \leq \lambda_n$ (it can be negative)
and
\begin{equation*}
 \xi_{-\varnothing}(\lambda) = \lambda ,
\mspace{40mu}
\xi_{+\varnothing}(\lambda) = (\lambda_2,\dotsc ,\dotsc,\lambda_n) .
\end{equation*}

For nonempty sequences $\und{t}^\ell = (t_1,\dotsc ,t_\ell)$ and $\und{i}^m=(i_1,\dotsc ,i_m)$  
we write  
$$\xi_{-\und{t}^{\ell}} = \xi_{-t_\ell}\circ \dotsm\circ\xi_{-t_2}\circ\xi_{-t_1}, \qquad
\xi_{+\und{i}^m} = \xi_{+i_m}\circ \dotsm\circ\xi_{+i_2}\circ\xi_{+i_1}$$
and
$$\xi_{-\und{t}^{\ell},+\und{i}^m} = \xi_{+\und{i}^m}\circ\xi_{-\und{t}_\ell} .$$

\begin{defn}\label{def:admB}
The map $\xi_{-\und{t}^\ell,+\und{i}^m}$ is \emph{$\lambda$-admissible} 
if 
$t_1\leq t_2 \leq \dotsc \leq t_\ell$ and 
$i_{m}\leq i_{m - 1} \leq \dotsc \leq i_1$  
and if
$\bigl(\xi_{-\und{t}^\ell}(\lambda),\xi^+_{-\und{t}^\ell,+\und{i}^m}(\lambda)\bigr)$ 
is in $\tau_{B_n}(\lambda)$.  
For fixed $k=m+\ell$ we denote by $\cD_{B_n,\lambda}^{k}$ the set of all such maps. 
\end{defn}
The set $\tau_{B_n}(\lambda)$ can be seen as the set 
$$\cD_{B_n,\lambda} = \bigsqcup_{k\geq 0}\cD_{B_n,\lambda}^k.$$ 
Hence, we can write $\xi_{(\mu,\nu)}$ for the corresponding $\xi_{-\und{t}^\ell,+\und{i}^m}$
without ambiguity.

\subsubsection{The Gelfand-Tsetlin basis}\label{ssec:GTB}

The complete Gelfand-Tsetlin pattern for $V_\lambda^{\soo}$
is
\begin{equation*}
\ket{\lambda^{(1)}, \nu^{(1)},\lambda^{(2)},\nu^{(2)}, \dotsc ,\lambda^{(n)},\nu^{(n)}}
\end{equation*}
where 
\begin{align*}
\lambda^{(i)} &= \bigl\{(\lambda_i^{(i)},\dotsc ,\lambda_n^{(i)})\;\vert\; 
\lambda_i^{(i)}\geq \lambda_{i+1}^{(i)}\geq \dotsm \geq \lambda_{n-1}^{(i)}\geq \lambda_n^{(i)}\geq 0 \bigr\}
\\[1ex]
\nu^{(i)} &= \bigl\{(\nu_i^{(i)},\dotsc ,\nu_n^{(i)})\;\vert\; 
\nu_i^{(i)}\geq \nu_{i+1}^{(i)}\geq \dotsm \geq \nu_{n-1}^{(i)}\geq \vert \nu_n^{(i)}\vert \; \bigr\}
\end{align*}
and where each triple $(\lambda^{(i)}, \nu^{(i)}, \lambda^{(i+1)})$ 
satisfies the inbetweenness condition~\eqref{eq:inbetB}.

Schematically we depict a complete GT-pattern as
\begin{equation*}
\left\vert 
\begin{array}{cccccccccc}
\lambda_1^{(1)} & & \lambda_2^{(1)} & & \lambda_3^{(1)} &   & & \lambda_n^{(1)} & & -\lambda_n^{(1)}
\\ 
& \nu_1^{(1)} & & \nu_2^{(1)} & & \dotsm  & & & \nu_n^{(1)} &
\\
 & & \lambda _2^{(2)} & & \lambda_3^{(2)} &   & & \lambda_n^{(2)} & & -\lambda_n^{(2)}
\\ 
&  & & \nu_2^{(2)} & & \dotsm &  & & \nu_n^{(2)} &
\\[1ex] 
 & & & & &  & & & \vdots  & 
\\[1ex]  
& &  & &  & &  & \lambda_n^{(n)} & & -\lambda_n^{(n)} 
\\ 
&  & &  & &  & & & \nu_n^{(n)} &
\end{array}
\right.\xy/r1.78pt/:
(0,37)*{}; (12,0)*{};  **\dir{-};(0,37)*{}; (12,0)*{};  **\dir{-};(0,37)*{}; (12,0)*{};  **\dir{-};
(12,0)*{}; (0,-37); **\dir{-}; (12,0)*{}; (0,-37); **\dir{-}; (12,0)*{}; (0,-37); **\dir{-}
\endxy  \ ,
\end{equation*}
with entries elements satisfying the inbetweenness relation~\eqref{eq:inbetB}.

Denote by $\cS(\lambda)$ the set of all complete Gelfand-Tsetlin patterns for  $V_\lambda^{\soo}$ 
and by $V_{\cS(\lambda)}$ the $\bQ(q)$-vector space spanned by $\cS(\lambda)$. 
There is an action of $\soo$ on  $V_{\cS(\lambda)}$ that makes it 
isomorphic to  $V_\lambda^{\soo}$ as $\soo$-modules. 
Explicit formulas for the $\soo$-action can be found for example in~\cite{molev}.

\subsection{Representations of quantum $\so$}\label{ssec:qalgD}

Type $D_n$ is very similar to type $B_n$. 
The QEA of type $D_n$ corresponds to the Dynkin diagram
\begin{equation*}
\labellist
\pinlabel $\dotsc$  at 188 60.5    
\tiny \hair 2pt 
\pinlabel $1$   at   7 36  
\pinlabel $2$   at  78 36 
\pinlabel $n-3$ at 282 35 
\pinlabel $n-2$ at 354 35 
\pinlabel $n-1$ at 468 98 
\pinlabel $n$   at 452 22 
\endlabellist
\figins{-19}{0.65}{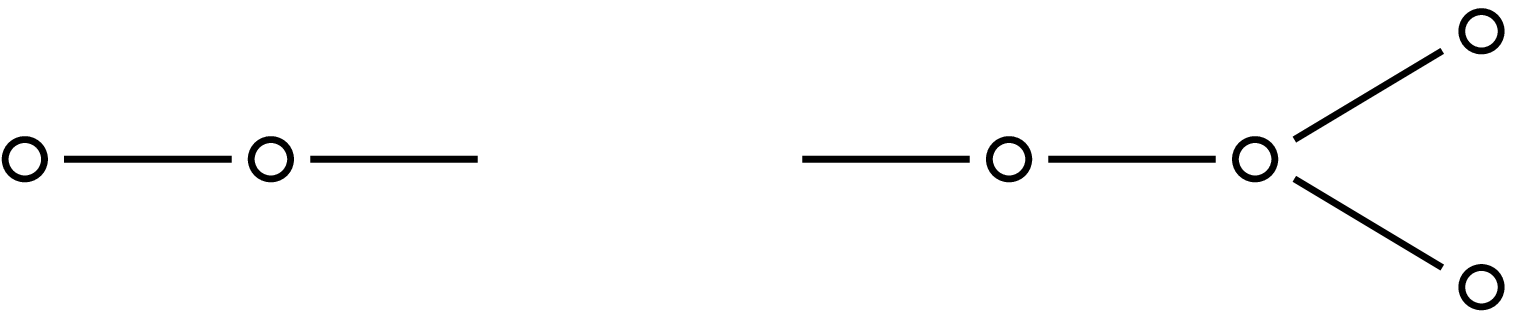} 
\end{equation*}
and has simple root system given by  
\begin{align*}
\alpha_1 = \varepsilon_1-\varepsilon_2,\ \dotsc ,\ \alpha_{n-1} = \varepsilon_{n-1}-\varepsilon_n,\ \alpha_n = \varepsilon_{n-1}+\varepsilon_n .
\end{align*}
Irreducible representations of $\so$ correspond to n-tuples 
$\lambda_1, \dotsc ,\lambda_n$ which are all integers (for vector representations) 
or all half-integers (for spin representations) and they satisfy
\begin{equation*}
\lambda_1 \geq \lambda_2 \geq \dotsm \geq \lambda_{n-1} \geq \vert\lambda_n\vert .
\end{equation*}

The correspondence with the set Dynkin labels is 
\begin{equation*}
\begin{cases}
\widebar{\lambda}_j = \lambda_j - \lambda_{j+1} & \text{ for } j=1,\dotsc ,n-1 , 
\\[2ex]
\widebar{\lambda}_n = \lambda_{n-1} + \lambda_{n} .
\end{cases}
\end{equation*}

\subsubsection{Branching rules}\label{ssec:brulesD}

As in the previous cases we let $\Gamma_n$ and $\Gamma_{n-1}$ be the Dynkin diagrams associated to types 
$D_n$ and $D_{n-1}$ 
with the vertices of $\Gamma_n$ labeled from 1 to $n$ and the vertices from $\Gamma_{n-1}$ 
labeled from 2 to $n$. 
Consider the inclusion $\Gamma_{n-1}\hookrightarrow\Gamma_n$ 
that adds a vertex labeled 1 at the beginning of $\Gamma_{n-1}$ and the corresponding edge, \emph{i.e.}   
\begin{equation}\label{eq:dynkinplusD}
\labellist
\pinlabel $\dotsc$  at 188 60.5    
\tiny \hair 2pt 
\pinlabel $2$   at   7 36  
\pinlabel $n-3$ at 208 35 
\pinlabel $n-2$ at 282 35 
\pinlabel $n-1$ at 396 98 
\pinlabel $n$   at 380 22 
\endlabellist
\figins{-19}{0.65}{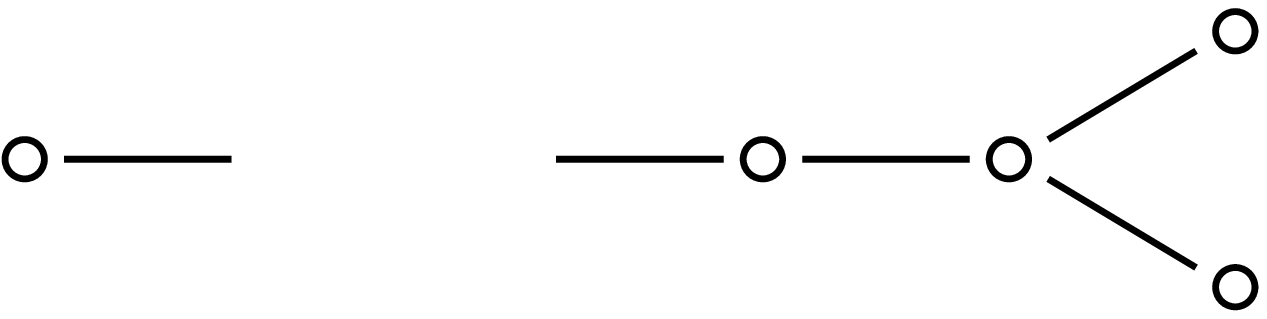}
\mspace{35mu}\hookrightarrow\mspace{35mu}  
\labellist
\pinlabel $\dotsc$  at 188 60.5    
\tiny \hair 2pt 
\pinlabel $1$   at   7 36  
\pinlabel $2$   at  78 36 
\pinlabel $n-3$ at 282 35 
\pinlabel $n-2$ at 354 35 
\pinlabel $n-1$ at 468 98 
\pinlabel $n$   at 452 22 
\endlabellist
\figins{-19}{0.65}{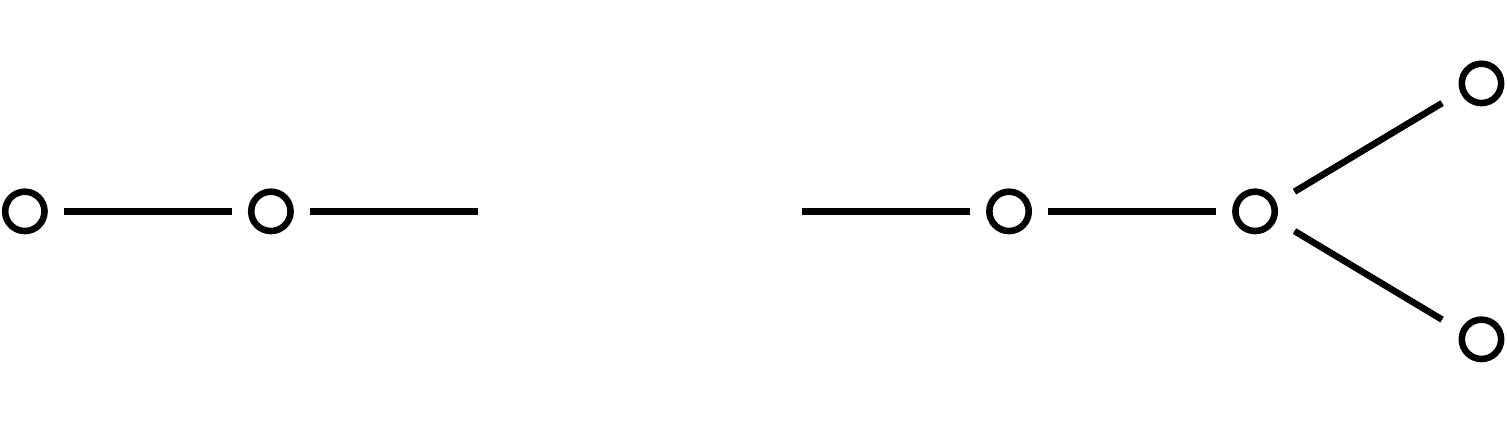} 
\end{equation}

As a quantum $\mathfrak{so}_{2n-2}$-representation we have an isomorphism,
\begin{equation}\label{eq:brDD}
\bigl(V_\lambda^{\so}\bigr)^{\mathfrak{so}_{2n-2}} \cong \bigoplus\limits_{\mu} c(\mu) V_\mu^{\mathfrak{so}_{2n-2}} ,
\end{equation}
where $\mu=(\mu_2,\dotsc ,\mu_n)$ is in $\Lambda_+^{\mathfrak{so}_{2n-1}}$
and
the multiplicity $c(\mu)$ is the number of 
$\nu = (\nu_2, \dotsc ,\nu_n)$ (all simultaneously integers or half-integers together with the $\lambda_i$s) satisfying 
\begin{align*}
\lambda_1 \geq \nu_2 \geq \lambda_2 \geq \dotsm \geq \lambda_{n-1} \geq \nu_n \geq \vert\lambda_n\vert
\\
\nu_2 \geq \mu_2 \geq \dotsm \geq \mu_{n-1} \geq \nu_n \geq \vert\mu_n\vert
\end{align*}
This can be depicted schematically as the \emph{Gelfand-Tsetlin pattern} 
of type $D_n$
\begin{equation}\label{eq:GTpattD}
\left\vert 
\begin{array}{ccccccccccc}
\lambda_1 & & \lambda _2 && \lambda_3 && \dotsm & && \lambda_n  
\\
& \nu_2 & & \nu_3 & &  \dotsm & & \dotsm  & \nu_n & 
\\
 & & \mu _2 & & \mu_3 & &  \dotsm & &  & \mu_n  
\end{array}
\right.\xy/r0.62pt/:
(0,37)*{}; (12,0)*{};  **\dir{-};(0,37)*{}; (12,0)*{};  **\dir{-};(0,37)*{}; (12,0)*{};  **\dir{-};
(12,0)*{}; (0,-37); **\dir{-}; (12,0)*{}; (0,-37); **\dir{-}; (12,0)*{}; (0,-37); **\dir{-}
\endxy \ ,
\end{equation}
where, for each triple of symbols $a_{i-1}$, $a_i$ and $a_{i+1}$ placed respectively in column 
$i-1$, $i$ and $i+1$, we have the inbetweenness condition
\begin{equation}\label{eq:inbetD}
a_{i-1} \geq a_i \geq \vert a_{i+1}\vert  .
\end{equation}
The set $\tau_{D_n}(\lambda)$ denotes the set of all such $(\nu,\mu)$.

Define the lattice maps 
$\xi_{+i}\colon \bZ^n\to \bZ^{n-1}$ and
$\xi_{-i}\colon \bZ^{n-1}\to \bZ^{n-1}$
by  
\begin{align*}
\xi_{+i}(\lambda) &= (\lambda_2,\dotsc ,\lambda_{i+1}+1,\dotsc,\lambda_n) 
\\
\xi_{-i}(\nu) &= (\nu_2,\dotsc ,\nu_{i}-1,\dotsc,\nu_n) 
\end{align*}
and
\begin{equation*}
 \xi_{-\varnothing}(\nu) = \nu ,
\mspace{40mu}
\xi_{+\varnothing}(\lambda) = (\lambda_2,\dotsc ,\lambda_n) .
\end{equation*}

For nonempty sequences $\und{t}^\ell = (t_1,\dotsc ,t_\ell)$ and $\und{i}^m=(i_1,\dotsc ,i_m)$  
we write  
$$\xi_{-\und{t}^{\ell}} = \xi_{-t_\ell}\circ \dotsm\circ\xi_{-t_2}\circ\xi_{-t_1}, \qquad
\xi_{+\und{i}^m} = \xi_{+i_m}\circ \dotsm\circ\xi_{+i_2}\circ\xi_{+i_1}$$
and
$$\xi_{-\und{t}^{\ell},+\und{i}^m} = \xi_{+\und{i}^m}\circ\xi_{-\und{t}_\ell} .$$
We use the same notational principle for the idempotents $p_{-\und{t}^\ell,+\und{i}^m}$.

\begin{defn}\label{def:admD}
The map $\xi_{+\und{i}^m,-\und{t}^\ell}$ is \emph{$\lambda$-admissible} 
if 
$i_1\leq i_2 \leq \dotsc \leq i_m$ and 
$t_{\ell}\leq t_{\ell - 1} \leq \dotsc \leq t_1$  
and if
$\bigl(\xi_{+\und{i}^m}(\lambda),\xi_{+\und{i}^m,-\und{t}^\ell}(\lambda)\bigr)$ 
is in $\tau_{D_n}(\lambda)$.  
For fixed $k=m+\ell$ we denote by $\cD_{D_n,\lambda}^{k}$ the set of all such maps. 
\end{defn}
The set $\tau_{D_n}(\lambda)$ can be seen as 
$$\cD_{D_n}=\bigsqcup_{k\geq 0}\cD_{D_n,\lambda}^k.$$ 
Hence, we can write $\xi_{(\mu,\nu)}$ for the corresponding $\xi_{+\und{}^m,-\und{t}^\ell}$ 
Without ambiguity.

\subsubsection{The Gelfand-Tsetlin basis}\label{ssec:GTD}

The complete Gelfand-Tsetlin pattern for the QEA of type $D_n$ is
\begin{equation*}
\ket{\lambda^{(1)}, \nu^{(2)},\lambda^{(2)},\nu^{(3)}, \dotsc ,\nu^{(n)}, \lambda^{(n)} }
\end{equation*}
where 
\begin{align*}
\lambda^{(i)} &= \bigl\{(\lambda_i^{(i)},\dotsc ,\lambda_n^{(i)})\;\vert\; 
\lambda_i^{(i)}\geq \lambda_{i+1}^{(i)}\geq \dotsm \geq \lambda_{n-1}^{(i)}\geq \vert\lambda_n^{(i)}\vert \; \bigr\}
\\[1ex]
\nu^{(i)} &= \bigl\{(\nu_i^{(i)},\dotsc ,\nu_n^{(i)})\;\vert\; 
\nu_i^{(i)}\geq \nu_{i+1}^{(i)}\geq \dotsm \geq \nu_{n-1}^{(i)}\geq \nu_n^{(i)}\geq 0\;  \bigr\} 
\end{align*}
and where each triple $(\lambda^{(i)}, \nu^{(i)}, \lambda^{(i+1)})$ 
satisfies the inbetweenness condition~\eqref{eq:inbetD}.

Schematically we depict a complete GT-pattern as
\begin{equation*}
\left\vert 
\begin{array}{cccccccccc}
\lambda_1^{(1)} & & \lambda_2^{(1)} & & \lambda_3^{(1)} &   & && \lambda_n^{(1)}  
\\ 
& \nu_2^{(2)} & & \nu_3^{(2)} & & \dotsm  & & \nu_n^{(2)}  &
\\
 & & \lambda _2^{(2)} & & \lambda_3^{(2)} &   && & \lambda_n^{(2)} 
\\ 
&  & & \nu_3^{(3)} & & \dotsm &  &  \nu_n^{(3)} &
\\[1ex] 
 & & & & &  & &  & \vdots &
\\[1ex]  
& & & & & & & \nu_n^{(n)} &
\\
& &  & &  & &  && \lambda_n^{(n)}  
\end{array}
\right.\xy/r1.78pt/:
(0,37)*{}; (12,0)*{};  **\dir{-};(0,37)*{}; (12,0)*{};  **\dir{-};(0,37)*{}; (12,0)*{};  **\dir{-};
(12,0)*{}; (0,-37); **\dir{-}; (12,0)*{}; (0,-37); **\dir{-}; (12,0)*{}; (0,-37); **\dir{-}
\endxy \ ,
\end{equation*}
with entries satisfying the inbetweeness relation~\eqref{eq:inbetD}. 

Denote $\cS(\lambda)$ the set of all complete Gelfand-Tsetlin patterns for  $V_\lambda^{\so}$ 
and by $V_{\cS(\lambda)}$ the $\bQ(q)$-vector space spanned by $\cS(\lambda)$. 
There is an action of $\so$ on  $V_{\cS(\lambda)}$
such that $V_{\cS(\lambda)}$ and $V_\lambda^{\so}$ are isomorphic 
as $\so$-modules. 
Explicit formulas for the $\so$-action can be found in~\cite{molev}.

\subsection{Representations of quantum $\spp$}\label{ssec:qalgC}

The QEA of type $C_n$ corresponds to the Dynkin diagram
\begin{equation*}
\labellist
\pinlabel $\dotsc$  at 188 60.5    
\tiny \hair 2pt 
\pinlabel $1$   at   7 36  
\pinlabel $2$   at  78 36 
\pinlabel $n-2$ at 282 35 
\pinlabel $n-1$ at 358 35 
\pinlabel $n$   at 433 35 
\endlabellist
\figins{-19}{0.65}{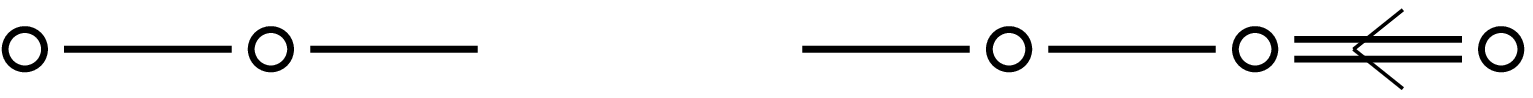} 
\end{equation*}
and has simple root system given by 
\begin{align*}
\alpha_1 = \varepsilon_1-\varepsilon_2,\ \dotsc ,\ \alpha_{n-1} = \varepsilon_{n-1}-\varepsilon_n,\ \alpha_n = 2\varepsilon_n .
\end{align*}
Irreducible representations of $\spp$ correspond to n-tuples 
$\lambda_1, \dotsc ,\lambda_n$ of integers 
satisfying 
\begin{equation*}
\lambda_1 \geq \lambda_2 \geq \dotsm \geq \lambda_{n-1} \geq \lambda_n \geq 0 ,
\end{equation*}
which correspond with the Dynkin labels as follows  
\begin{equation*}
\begin{cases}
\widebar{\lambda}_j = \lambda_j - \lambda_{j+1} & \text{ for } j=1,\dotsc ,n-1 ,
\\[2ex]
\widebar{\lambda}_n = \lambda_{n} .
\end{cases}
\end{equation*}

\subsubsection{Branching rules}\label{ssec:brulesC}

As in the previous case we let $\Gamma_n$ and $\Gamma_{n-1}$ be the Dynkin diagrams associated to types 
$C_n$ and $C_{n-1}$ 
with the vertices of $\Gamma_n$ labeled from 1 to $n$ and the vertices from $\Gamma_{n-1}$ 
labeled from 2 to $n$. 
Consider the inclusion $\Gamma_{n-1}\hookrightarrow\Gamma_n$ 
that adds a vertex labeled 1 at the beginning of $\Gamma_{n-1}$ and the corresponding edge, \emph{i.e.}   
\begin{equation}\label{eq:dynkinplusC}
\labellist
\pinlabel $\dotsc$  at 118 60.5    
\tiny \hair 2pt 
\pinlabel $2$   at   7 36  
\pinlabel $n-2$ at 215 35 
\pinlabel $n-1$ at 288 35 
\pinlabel $n$   at 364 35 
\endlabellist
\figins{-19}{0.65}{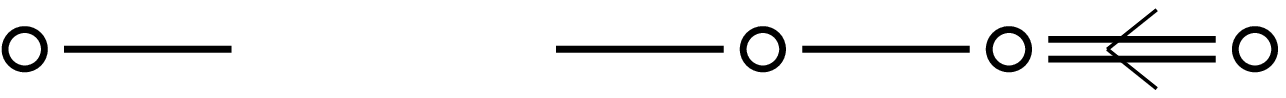}
\mspace{35mu}\hookrightarrow\mspace{35mu}  
\labellist
\pinlabel $\dotsc$  at 188 60.5    
\tiny \hair 2pt 
\pinlabel $1$   at   7 36  
\pinlabel $2$   at  78 36 
\pinlabel $n-2$ at 287 35 
\pinlabel $n-1$ at 364 35 
\pinlabel $n$   at 435 35 
\endlabellist
\figins{-19}{0.65}{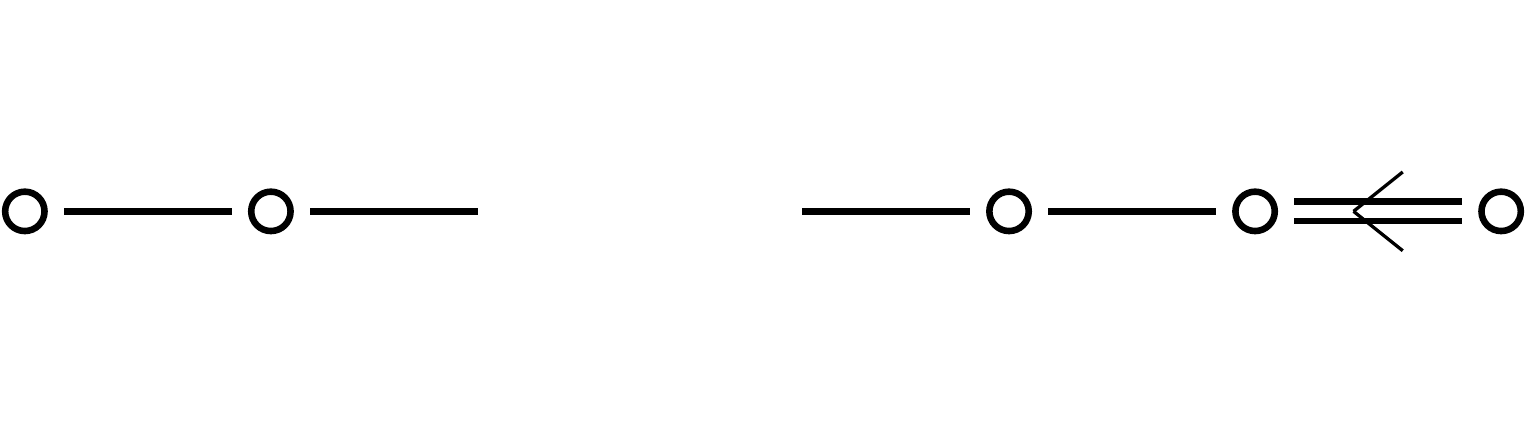} 
\end{equation}

As a quantum $\mathfrak{sp}_{2n-2}$-representation,
\begin{equation}\label{eq:brCC}
\bigl(V_\lambda^{\spp}\bigr)^{\mathfrak{sp}_{2n-2}} \cong \bigoplus\limits_{\mu} c(\mu) V_\mu^{\mathfrak{sp}_{2n-2}} ,
\end{equation}
where $\mu=(\mu_2,\dotsc ,\mu_n)$ is in $\Lambda_+^{\mathfrak{sp}_{2n-2}}$
and
the multiplicity $c(\mu)$ is the number of 
$\nu = (\nu_1,\dotsc ,\nu_n)$ satisfying 
\begin{align*}
\lambda_1 \geq \nu_1 \geq \lambda_2 \geq \dotsm \geq \nu_{n-1} \geq \lambda_n \geq \nu_n \geq 0
\\
\nu_1 \geq \mu_2 \geq \nu_2\geq \dotsm \geq\nu_{n-1}\geq \mu_{n} \geq \nu_n \geq 0
\end{align*}
This can be depicted schematically as the \emph{Gelfand-Tsetlin pattern} 
of type $C_n$
\begin{equation}\label{eq:GTpattC}
\left\vert
\begin{array}{cccccccccc}
\lambda_1 & & \lambda _2 & & \lambda_3 & & \dotsm & & \lambda_n &
\\
& \nu_1 & & \nu_2 & &\dotsm && \dotsm   & & \nu_n 
\\
 & & \mu _2 & & \mu_3 & & \dotsm & & \mu_n & 
\end{array}
\right.\xy/r0.62pt/:
(0,37)*{}; (12,0)*{};  **\dir{-};(0,37)*{}; (12,0)*{};  **\dir{-};(0,37)*{}; (12,0)*{};  **\dir{-};
(12,0)*{}; (0,-37); **\dir{-}; (12,0)*{}; (0,-37); **\dir{-}; (12,0)*{}; (0,-37); **\dir{-}
\endxy \ ,
\end{equation}
where, for each triple of symbols $a_{i-1}$, $a_i$ and $a_{i+1}$ placed respectively in column 
$i-1$, $i$ and $i+1$, we have the inbetweenness condition
\begin{equation}\label{eq:inbetC}
a_{i-1} \geq a_i \geq a_{i+1} .
\end{equation}
The set $\tau_{C_n}(\lambda)$ denotes the set of all such $(\nu,\mu)$. 
We stress that the $\nu$s appearing above are not to be be interpreted as  
$\spp$-highest weights.

\smallskip

In the sequel it will be useful to interpret the GT-pattern in~\eqref{eq:GTpattC} as 
a sequence of maps of 
$\mathfrak{sp}_{2n-2}$-representations  as follows:  
\begin{equation*}
V^{\spp}_\lambda 
\xra{\phi_{\lambda_1-\nu_1}} V^{\spp}_{(\nu_1,\lambda_2,\dotsc,\lambda_n)} 
\xra{\phi_{\nu}} V^{\spp}_{(\nu_1,\nu_2,\dotsc,\nu_n)}  
\xra{\phi_{\mu}} V^{\mathfrak{sp}_{2n-2}}_{(\mu_2,\dotsc,\mu_n)}  .
\end{equation*}

Define the lattice maps 
$\xi_{-1}\colon \bZ^n\to \bZ^n$, 
$\xi_{-i}\colon \bZ^n\to \bZ^{n-1}$ ($i>1$),
and 
 $\xi_{-i}\colon \bZ^{n-1}\to \bZ^{n-1}$
by
\begin{align*}
\xi_{-i}(\lambda) &= 
\begin{cases} 
(\lambda_1-1, \lambda_2, \dotsc ,\lambda_n) 
\\[0.5ex] 
(\lambda_2, \dotsc ,\lambda_{i}-1,\dotsc,\lambda_n) 
\end{cases} 
& \begin{matrix}  i = 1 \\[1.2ex] 1<i\leq n \end{matrix}
\\[1ex]
\ \ \xi_{+i}(\nu) & =   (\nu_2, \dotsc ,\nu_{i+1}+1,\dotsc,\nu_n)    & 1 \leq i < n
\end{align*}
and
\begin{equation*}
 \xi_{-\varnothing}(\lambda) = (\lambda_2,\dotsc ,\dotsc,\lambda_n) ,
\mspace{40mu}
\xi_{+\varnothing}(\nu) = \nu .
\end{equation*}

For nonempty sequences $\und{i}^m=(i_1,\dotsc ,i_m)$  and
$\und{t}^\ell = (t_1,\dotsc ,t_\ell)$ with $t_j\neq 1$ for $j=1,\dotsc ,\ell$ 
we write  
\begin{equation*} 
\xi_{+\und{i}^m} = \xi_{+i_m}\circ \dotsm\circ\xi_{+i_2}\circ\xi_{+i_1}, 
\qquad
\xi_{-\und{t}^{\ell}} = \xi_{-t_\ell}\circ \dotsm\circ\xi_{-t_2}\circ\xi_{-t_1}
\end{equation*}
and
$$\xi_{-\und{t}^{\ell},+\und{i}^m} = \xi_{+\und{i}^m}\circ\xi_{-\und{t}_\ell} .$$

For $0 \leq c\leq \lambda_1-\lambda_2$ we write
$\tau_{C_n}(\lambda,c) = \tau_{C_n}(\lambda)\vert_{\nu_1 =\lambda_1-c}$ 

\begin{defn}\label{def:admC}
The map  $\xi_{-\und{t}^\ell,+\und{i}^m}$ is said to be \emph{$\lambda c$-admissible} 
$1<t_1\leq t_2 \leq \dotsc \leq t_\ell$ and 
$i_{m}\leq i_{m - 1} \leq \dotsc \leq i_1$  
and if
$\bigl( (\lambda_1-c) \times \xi_{-\und{t}^\ell}(\lambda),\xi_{-\und{t}^\ell,+\und{i}^m}(\lambda)\bigr)$ 
is in $\tau_{C_n}(\lambda,c)$.  
For fixed $k=m+\ell$ and $c$ we denote by $\cD_{C_n,\lambda,c}^{k}$ the set of all such maps. 
\end{defn}
Here $(a) \times (b_1,\dotsc ,b_n)=(a, b_1,\dotsc ,b_n)$ is the concatenation of sequences. 
The set $\tau_{C_n}(\lambda)$ can then be seen as the set 
$$\cD_{C_n}=\bigsqcup\limits_{k\geq 0}\bigsqcup\limits_{c=0}^{\lambda_1-\lambda_2}\cD_{C_n,\lambda,c}^k.$$ 
Hence we can write $\xi_{(\mu,\nu)}$ for the corresponding $\xi_{-\und{t}^{\ell-c},+\und{i}^m}$
without ambiguity.

Each  $\xi_{-\und{t}^\ell+\und{i}^m}$ in $\cD_{\lambda}$ ($\lambda$-admissible) 
is of the form $\xi_{-\und{t}^{\ell-c},+\und{i}^m}\circ\xi_{-1}^c$
where $\xi_{-1}^c$ is $\xi_{-1}\circ\dotsm\circ \xi_{-1}$ ($c$ times) 
and $\xi_{-\und{t}^{\ell-c},+\und{i}^m}$ is in $\cD_{\lambda,c}$.

\subsubsection{The Gelfand-Tsetlin basis}\label{ssec:GTC}

The complete Gelfand-Tsetlin pattern for the QEA of type $C_n$ is
\begin{equation*}
\ket{\lambda^{(1)}, \nu^{(1)},\lambda^{(2)},\nu^{(2)}, \dotsc ,\lambda^{(n)},\nu^{(n)}}
\end{equation*}
where 
\begin{align*}
\lambda^{(i)} &= \bigl\{(\lambda_i^{(i)},\dotsc ,\lambda_n^{(i)})\;\vert\; 
\lambda_i^{(i)}\geq \lambda_{i+1}^{(i)}\geq \dotsm \geq \lambda_{n-1}^{(i)}\geq \lambda_n^{(i)}\geq 0\; \bigr\}
\\[1ex]
\nu^{(i)} &= \bigl\{(\nu_i^{(i)},\dotsc ,\nu_n^{(i)})\;\vert\; 
\nu_i^{(i)}\geq \nu_{i+1}^{(i)}\geq \dotsm \geq \nu_{n-1}^{(i)}\geq \nu_n^{(i)}\geq 0\;  \bigr\}
\end{align*}
and where each triple $(\lambda^{(i)}, \nu^{(i)}, \lambda^{(i+1)})$ 
satisfies the inbetweenness condition~\eqref{eq:inbetC}.

Schematically we depict a complete GT-pattern as
\begin{equation*}
\left\vert 
\begin{array}{ccccccccc}
\lambda_1^{(1)} & & \lambda_2^{(1)} & & \lambda_3^{(1)} &   & & \lambda_n^{(1)} & 
\\ 
& \nu_1^{(1)} & & \nu_2^{(1)} & & \dotsm  & & & \nu_n^{(1)} 
\\
 & & \lambda _2^{(2)} & & \lambda_3^{(2)} &   & & \lambda_n^{(2)} & 
\\ 
&  & & \nu_2^{(2)} & & \dotsm &  & & \nu_n^{(2)} 
\\[1ex] 
 & & & & &  & & & \vdots   
\\[1ex]  
& &  & &  & &  & \lambda_n^{(n)} & 
\\ 
&  & &  & &  & & & \nu_n^{(n)} 
\end{array}
\right.\xy/r1.78pt/:
(0,37)*{}; (12,0)*{};  **\dir{-};(0,37)*{}; (12,0)*{};  **\dir{-};(0,37)*{}; (12,0)*{};  **\dir{-};
(12,0)*{}; (0,-37); **\dir{-}; (12,0)*{}; (0,-37); **\dir{-}; (12,0)*{}; (0,-37); **\dir{-}
\endxy  \ , 
\end{equation*}
with entries satisfying the inbetweenness relation~\eqref{eq:inbetC}.

As before we denote by $\cS(\lambda)$ the set of all complete Gelfand-Tsetlin patterns for  $V_\lambda^{\spp}$ 
and by $V_{\cS(\lambda)}$ the $\bQ(q)$-vector space spanned by $\cS(\lambda)$. 
There is an actio of $\spp$ on $V_{\cS(\lambda)}$ such that 
the $\spp$-modules $V_{\cS(\lambda)}$ and $V_\lambda^{\spp}$ are isomorphic. 
Explicit formulas for the $\spp$-action can be found in~\cite{molev}.

\section{KLR algebras and their cyclotomic quotients}\label{sec:KLRalgebras}

In this section we describe the quiver Hecke algebras which were introduced by Khovanov and Lauda in~\cite{KL1,KL2} 
and independently by Rouquier in~\cite{Rouq1}. We concentrate on the particular case of types $B_n$, 
$C_n$ and $D_n$.    
The KLR algebra $R_{\Gamma}$ associated to the quiver $\Gamma$ 
is the algebra generated by $\Bbbk$-linear combinations of isotopy classes of 
braid-like planar diagrams where each strand is labeled by a simple root of $\mathfrak{g}$. 
Strands can intersect transversely to form crossings and they can also carry dots. 
Multiplication is given by concatenation of diagrams and the collection of such diagrams is subject 
to relations~\eqref{eq:R2}-\eqref{eq:dotslide} below  
(for the sake of simplicity we write $i$ instead $\alpha_i$ when labeling a strand). 
We read, by convention, diagrams from bottom to top (and from left to right) and therefore the diagram 
for the product $a.b$ is the diagram obtained by stacking the diagram for $a$ 
on the top of the one for $b$. 
\begin{equation}\label{eq:R2}
\labellist
\tiny \hair 2pt
\pinlabel $i$ at -5 0
\pinlabel $j$ at 61 -2
\endlabellist
\figins{-23}{0.7}{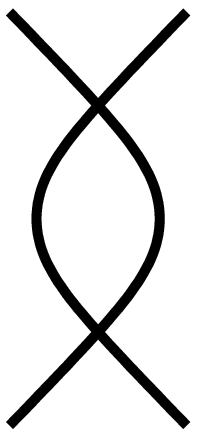}\ \
=\ \
\begin{cases}
\allowdisplaybreaks
\mspace{35mu}0 
&\mspace{20mu}\text{ if }\ \ i = j\ ,  
\\[1ex]
\mspace{20mu}
\labellist
\tiny \hair 2pt
\pinlabel $i$ at -5 0
\pinlabel $j$ at 61 -2
\endlabellist 
\figins{-23}{0.7}{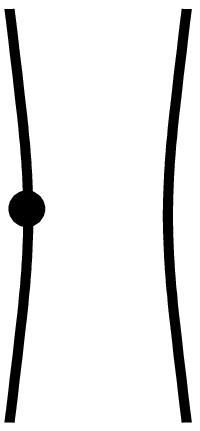}
\quad +\quad 
\labellist
\tiny \hair 2pt
\pinlabel $i$ at -5 0
\pinlabel $j$ at 61 -2
\endlabellist 
\figins{-23}{0.7}{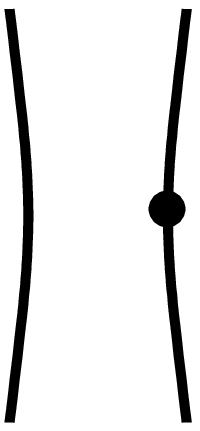}
&\mspace{20mu}\text{if }\ \ 
\labellist
\tiny \hair 2pt
\pinlabel $i$  at  4 22
\pinlabel $j$  at 70 20
\endlabellist
\figins{-16}{0.55}{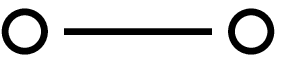}\ ,
\\[6ex]
\mspace{20mu}
\labellist
\tiny \hair 2pt
\pinlabel $i$ at -5 0
\pinlabel $j$ at 61 -2
\endlabellist 
\figins{-23}{0.7}{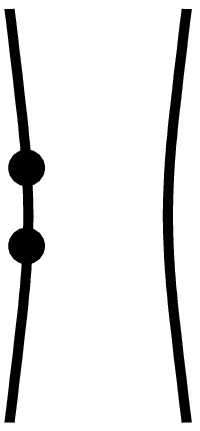}
\quad +\quad 
\labellist
\tiny \hair 2pt
\pinlabel $i$ at -5 0
\pinlabel $j$ at 61 -2
\endlabellist 
\figins{-23}{0.7}{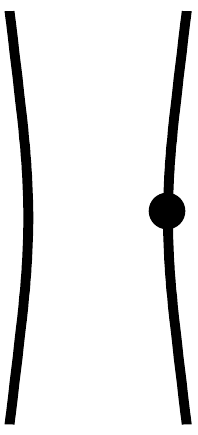}
&\mspace{20mu}\text{if }\ \ 
\labellist
\tiny \hair 2pt
\pinlabel $i$  at  4 22
\pinlabel $j$  at 70 20
\endlabellist
\figins{-16}{0.55}{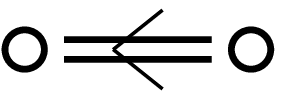}\ ,
\\[6ex]
\mspace{20mu}
\labellist
\tiny \hair 2pt
\pinlabel $i$ at -5 0
\pinlabel $j$ at 61 -2
\endlabellist 
\figins{-23}{0.7}{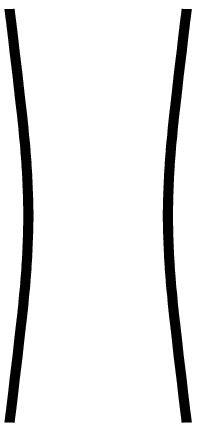}
&\mspace{20mu}\text{else}\ ,
\end{cases} 
\end{equation}

\medskip

%
\begin{align}
\label{eq:dotslide}
\labellist
\tiny \hair 2pt
\pinlabel $i$ at -5   0
\pinlabel $j$ at 97  -2
\endlabellist
\figins{-20}{0.60}{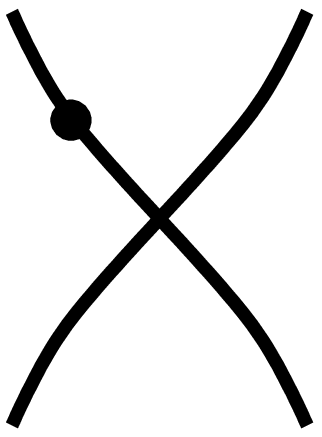}
\ \ - \ \
\labellist
\tiny \hair 2pt
\pinlabel $i$ at -5   0
\pinlabel $j$ at 97  -2
\endlabellist
\figins{-20}{0.60}{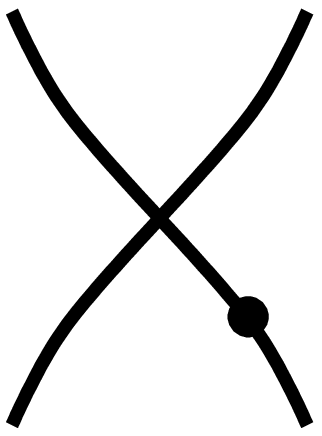}
\ \ = \ \ 
\delta_{ij}\ \ \   
\labellist
\tiny \hair 2pt
\pinlabel $i$ at -5   0
\pinlabel $j$ at 97  -2
\endlabellist
\figins{-20}{0.60}{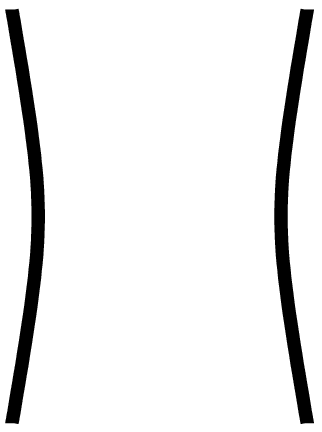}
\quad = \quad  
\labellist
\tiny \hair 2pt
\pinlabel $i$ at -5   0
\pinlabel $j$ at 97  -2
\endlabellist
\figins{-20}{0.60}{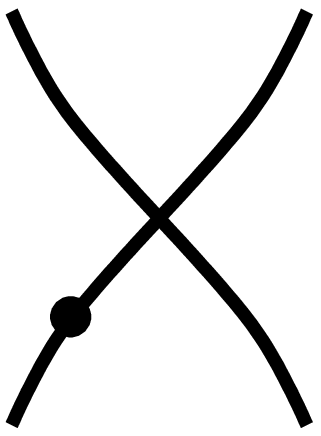}
\ \ - \ \
\labellist
\tiny \hair 2pt
\pinlabel $i$ at -5   0
\pinlabel $j$ at 97  -2
\endlabellist
\figins{-20}{0.60}{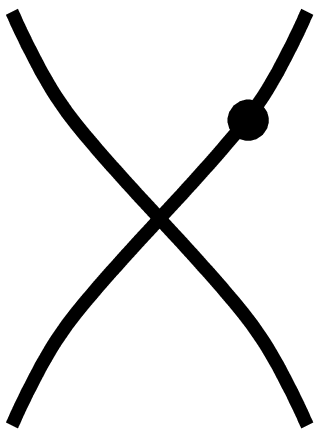}
\vspace*{4ex}
\end{align}

\medskip
%
\begin{equation}\label{eq:R3}
\labellist
\tiny \hair 2pt
\pinlabel $i$ at  -5 0
\pinlabel $j$ at  72 -2
\pinlabel $k$ at 132 -2
\endlabellist
\figins{-23}{0.7}{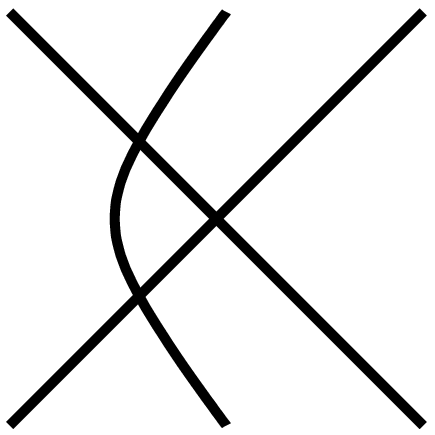}
\ \ - \ \
\labellist
\tiny \hair 2pt
\pinlabel $i$ at  -5 0
\pinlabel $j$ at  52 -2
\pinlabel $k$ at 132 -2
\endlabellist
\figins{-23}{0.7}{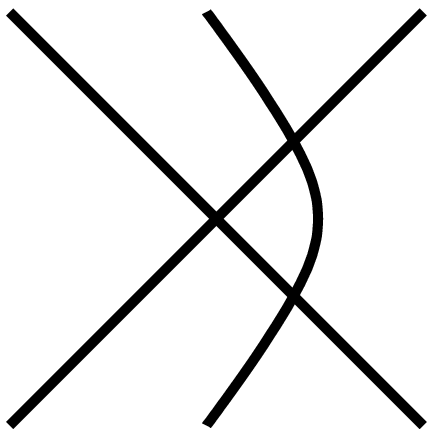}
\ \ = \ \ \ 
\begin{cases}
\mspace{20mu} 
\labellist
\tiny \hair 2pt
\pinlabel $i$ at  -5 0
\pinlabel $j$ at  68 -2
\pinlabel $k$ at 124 -2
\endlabellist
\figins{-23}{0.7}{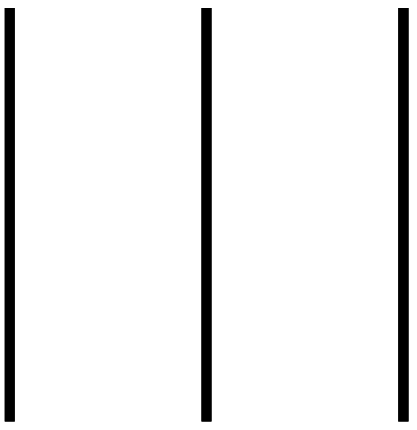} & \text{if }\ \ i=k\text{ and }\ \ \  
\substack{
\labellist
\tiny \hair 2pt
\pinlabel $i$  at  4 22
\pinlabel $j$  at 70 20
\endlabellist
\figins{-16}{0.55}{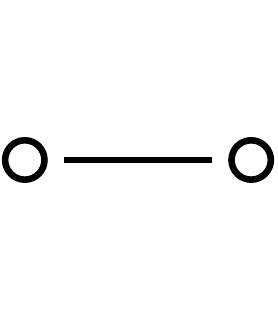}\ ,
\\
\labellist
\pinlabel $\text{or}$ at 39 88
\tiny \hair 2pt
\pinlabel $i$  at  4 22
\pinlabel $j$  at 70 20
\endlabellist
\figins{-16}{0.55}{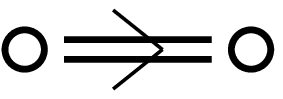}\ ,
}
\\[8ex]
\mspace{20mu} 
\labellist
\tiny \hair 2pt
\pinlabel $i$ at  -5 0
\pinlabel $j$ at  68 -2
\pinlabel $k$ at 124 -2
\endlabellist
\figins{-23}{0.7}{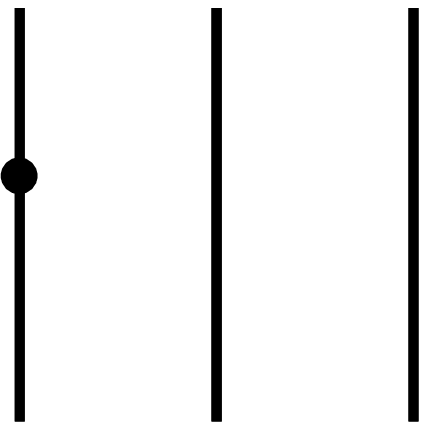}
\ \ + \ \  
\labellist
\tiny \hair 2pt
\pinlabel $i$ at  -5 0
\pinlabel $j$ at  68 -2
\pinlabel $k$ at 124 -2
\endlabellist
\figins{-23}{0.7}{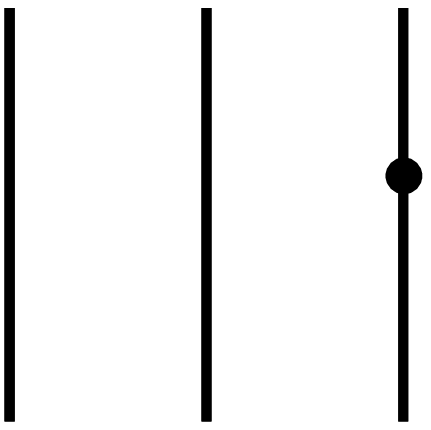}
& \text{if }\ \ i=k\text{ and }\ \ 
\labellist
\tiny \hair 2pt
\pinlabel $i$  at  4 22
\pinlabel $j$  at 70 20
\endlabellist
\figins{-16}{0.55}{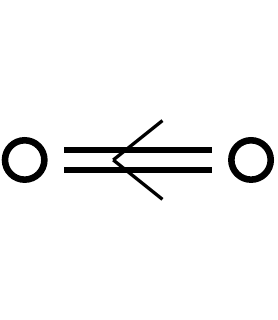}\ ,
\\[8ex]
\mspace{55mu}
0 & \text{else}\ ,
\end{cases}
\end{equation}

\medskip

Algebra $R_{\Gamma}$ is graded with the degrees given by 
\begin{align}
\deg\biggl(\  
\labellist
\tiny \hair 2pt
\pinlabel $i$ at -5   0
\pinlabel $j$ at 97  -2
\endlabellist
\figins{-10}{0.40}{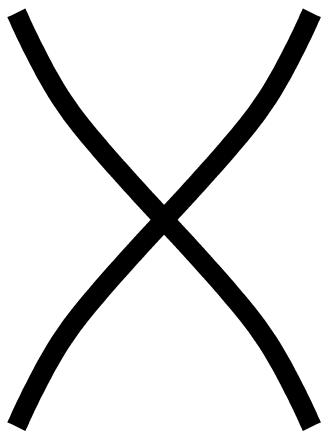}\ 
\biggr) = -d_ic_{ij}
\mspace{60mu}
\deg\biggl(\  
\labellist
\tiny \hair 2pt
\pinlabel $i$ at -5   0
\endlabellist
\figins{-10}{0.40}{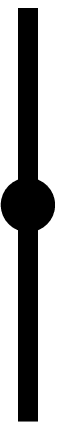}\ 
\biggr) = 2d_{i} .
\vspace*{4ex}
\end{align}

\medskip

\n The following useful relation follows from~\eqref{eq:dotslide} and will be used in the sequel.
\begin{align}
\label{eq:dotslides}
\labellist
\tiny \hair 2pt
\pinlabel $i$ at -5 0  \pinlabel $i$ at 97 0 \pinlabel $d$ at 0 90
\endlabellist
\figins{-20}{0.60}{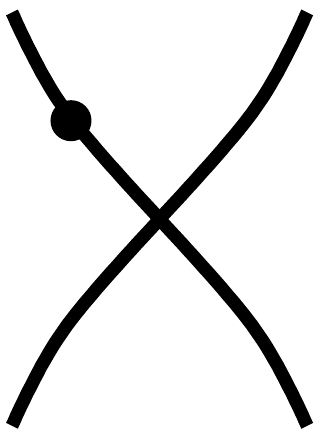}
\ \ - \ \
\labellist
\tiny \hair 2pt
\pinlabel $i$ at -5 0 \pinlabel $i$ at 97 0 \pinlabel $d$ at 90 35
\endlabellist
\figins{-20}{0.60}{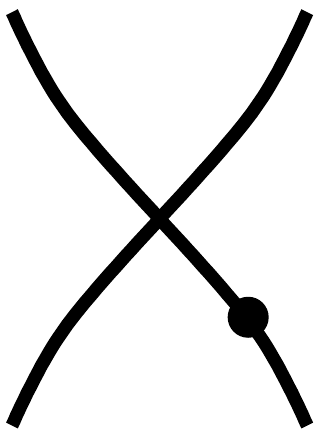}
\ \ = \ \ 
\sum_{\ell_1+\ell_2 = d-1}\ \ \   
\labellist
\tiny \hair 2pt
\pinlabel $i$ at -5 0 \pinlabel $i$ at 97 0 
\pinlabel $\ell_1$ at 32 70 \pinlabel $\ell_2$ at 103 70
\endlabellist
\figins{-20}{0.60}{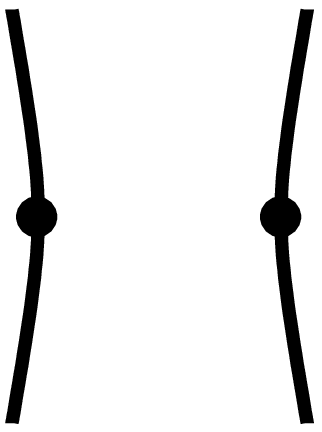}
\quad = \quad  
\labellist
\tiny \hair 2pt
\pinlabel $i$ at -5 0 \pinlabel $i$ at 97 0 \pinlabel $d$ at 0 35
\endlabellist
\figins{-20}{0.60}{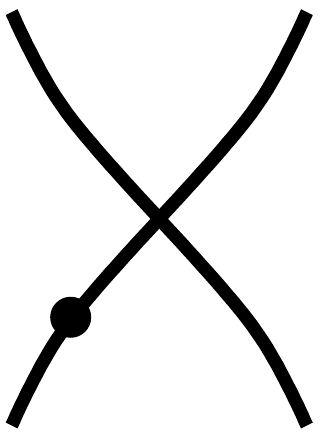}
\ \ - \ \
\labellist
\tiny \hair 2pt
\pinlabel $i$ at -5 0 \pinlabel $i$ at 97 0 \pinlabel $d$ at 90 90
\endlabellist
\figins{-20}{0.60}{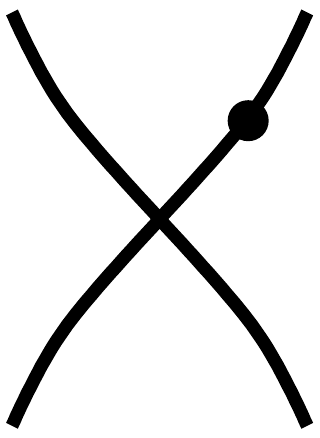}
\vspace{4ex}
\end{align}

\medskip

For $\beta = \sum\limits_{j\in\Gamma}\beta_i \alpha_i$ let $R_{\Gamma}(\beta)$ be the subalgebra generated by all 
elements 
of $R_{\Gamma}$ containing exactly $\beta_i$ strands labeled $i$.
We have  
\begin{equation*}
R_{\Gamma} = \sum\limits_{\beta\in\Lambda_+(\g)}R_\Gamma(\beta) .
\end{equation*}
We also denote by
\begin{equation*}
R_{\Gamma}(k\alpha_1) = \bigoplus\limits_{\beta'\in\Lambda_+(\g)/\bZ\alpha_1} R_{\Gamma}(\beta' + k\alpha_1) 
\end{equation*}
the subalgebra of $R_{\Gamma}$ containing exactly $k$ strands labeled $1$. 
With this notation we have 
\begin{equation}\label{eq:decKLR1}
R_{\Gamma} = \bigoplus\limits_{k\geq 0}R_{\Gamma}(k\alpha_1)  .
\end{equation}
For a sequence $\und{i}=(i_1,\dotsc , i_k)$ with $i_j$ corresponding to the simple root $\alpha_{i_j}$ we write 
$1_{\und{i}}$ for the idempotent formed by $k$ vertical strands with labels 
in the order given by $\und{i}$, that is 
\begin{equation*}
1_{\und{i}}\ = 
\labellist
\pinlabel $\dotsc$ at  145 65   
\tiny \hair 2pt
\pinlabel $i_1$   at   2 12
\pinlabel $i_2$   at  46 12
\pinlabel $i_3$   at  88 12
\pinlabel $i_k$   at 198 12
\endlabellist
\mspace{15mu}
\figins{-34}{0.9}{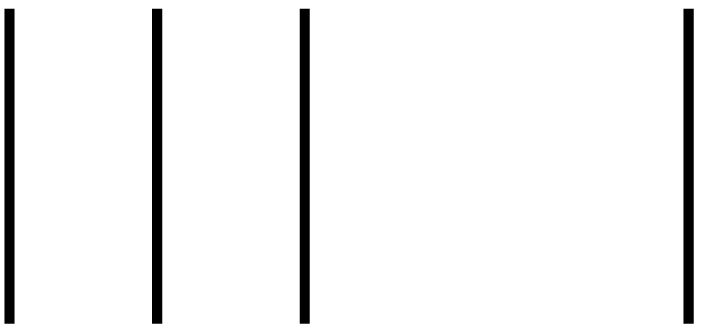} 
\end{equation*}
%
We write $1_{*\ell *}$ for $1_{\und{i}}$ if the sequence of 
labels $\und{i}=\und{j}\und{\ell}\und{j}'$ can be written as a concatenation of certain sequences and  
we are only interested in the $\und{\ell}$ part and the $*$ can be arbitrary.
We also write $x_{r,\und{i}}$ for the diagram consisting of a dot on the $r$-th strand of $1_{\und{i}}$,
\begin{equation*}
x_{r,\und{i}}\ = 
\labellist
\pinlabel $\dotsc$ at   45 50   
\pinlabel $\dotsc$ at  135 50   
\tiny \hair 2pt
\pinlabel $i_1$   at   7 -10
\pinlabel $i_r$   at  92 -10
\pinlabel $i_k$   at 178 -10
\endlabellist
\mspace{15mu}
\figins{-19}{0.6}{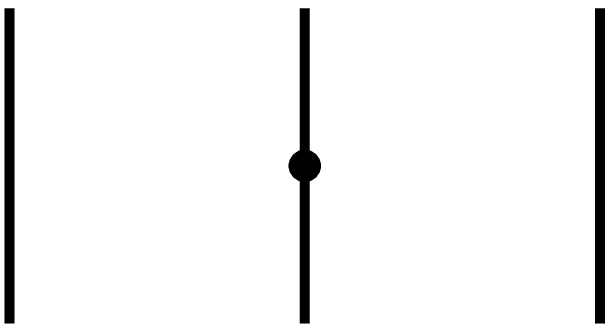}
\end{equation*}

\medskip

For $\beta$ as above we denote by $\seq(\beta)$ the set of all sequences $\und{i}$ of simple roots 
in which $i_j$ appears exactly $\beta_j$ times. 
The identity of $R_{\Gamma}(\beta)$ is then given by 
\begin{equation*}
1_{R_{\Gamma}(\beta)} =
\sum_{\und{i}\:\in\:\seq(\beta)}
\labellist
\pinlabel $\dotsc$ at  145 65   
\tiny \hair 2pt
\pinlabel $i_1$   at   2 12
\pinlabel $i_2$   at  46 12
\pinlabel $i_3$   at  88 12
\pinlabel $i_k$   at 198 12
\endlabellist
\mspace{15mu}
\figins{-34}{0.9}{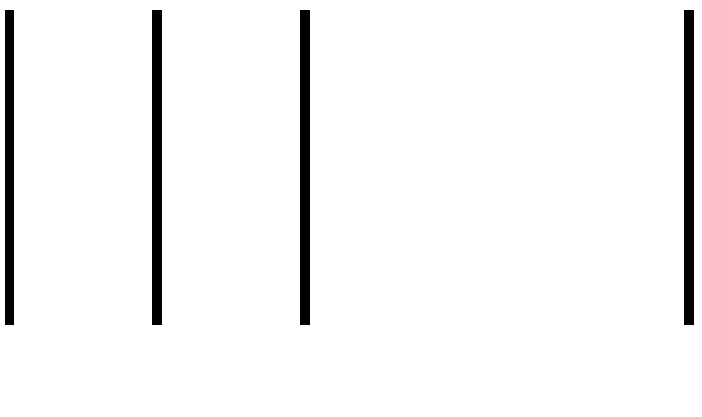} 
\end{equation*}

\medskip

For an idempotent $e\in R_{\Gamma}$ there is a (right) projective module $_{e}P=eR_{\Gamma}$. 
For $e=1_{\und{i}}$ this is the projective spanned by all diagrams whose labels  
end up in the sequence $\und{i}$.  The definition of the left projective $P_e$ is similar. 

Denote by $R_{\Gamma}\amod$ and $R_{\Gamma}\prmod$ the categories of finitely generated, graded, right $R_{\Gamma}$-modules
and of finitely generated, graded, projective right $R_{\Gamma}$-modules, respectively. 
For idempotents $e$, $e'$ we have 
\begin{equation*}
\Hom_{R_{\Gamma}\amod}\bigl({}_{e}P, ~{}_{e'\!}P \bigr) = e'R_{\Gamma}e  . 
\end{equation*}

\medskip

For a graded algebra $A$ we denote by $K'_0(A)$ the Grothendieck group of finitely generated, graded, projective 
$A$-modules and write $K_0(A)$ for $\bQ(q)\otimes_{\bZ[q,q^{-1}]} K'_0(A)$.  

\smallskip

\begin{thm}[Khovanov-Lauda~\cite{KL1}, Rouquier~\cite{Rouq1}] 
The Grothendieck $K_0(R_{\Gamma})$ is isomorphic to the lower half 
$U_q^-(\g)$ through the map that takes $[{}_{\und{i}}P]$ to $F_{\und{i}}$. 
\end{thm}

\subsection{Categorical inclusion and projection for KLR algebras}

Le $\Gamma_n$ and $\Gamma_{n-1}$ the Dynkin diagrams associated to $\g_n$ and $\g_{n-1}$ 
respectively and consider the inclusion $\Gamma_{n-1}\hookrightarrow\Gamma_n$ 
that adds a vertex and the corresponding edge at the beginning of $\Gamma_{n-1}$, 
like the ones in Equations~\eqref{eq:dynkinplusB},~\eqref{eq:dynkinplusD} and~\eqref{eq:dynkinplusC}.   
This induces an inclusion of KLR algebras 
\begin{equation*}
\imath\colon R_{\g_{n-1}}\hookrightarrow R_{\g_{n}}
\mspace{50mu}
x\mapsto x
\end{equation*}
which coincides with the obvious map coming from the decomposition
\begin{equation}\label{eq:decKLR2}
R_{\g_{n}} = \bigoplus\limits_{k\geq 0}R_{\g_{n}}(k\alpha_1) 
\cong 
R_{\g_{n-1}} + \bigoplus\limits_{k\geq 1}R_{\g_{n}}(k\alpha_1). 
\end{equation} 
As in the case of $\g=\sll$ (see~\cite{vaz2}) the functors of inclusion and restriction induced by $\imath$ 
\begin{align*}
\Ind_\imath \colon R_{\g_{n-1}}\amod \to R_{\g_n}\amod
\mspace{60mu}
\Res_\imath \colon R_{\g_n}\amod \to R_{\g_{n-1}}\amod
\end{align*}
are biadjoint, take projectives to projectives 
and descend to the natural inclusion and projection maps between the Grothendieck groups.


\subsection{Cyclotomic KLR-algebras}\label{ssec:cycKLR}

Fix a $\g$-highest weight 
$\lambda$ from now on and 
let $I^{\lambda}$ the two-sided ideal of $R_{\g}$ genera\-ted by $x^{\llambda_{i_1}}_{1,\und{i}}$ 
for all sequences $\und{i}$. 

\begin{defn}
The cyclotomic KLR algebra $R_{\g}^\lambda$ is the quotient of 
$R_{\g}$ by the two-sided ideal $I^{\lambda}$. 
\end{defn}
In terms of diagrams $R_{\g}^\lambda$ is the quotient by the two-sided ideal generated by all the diagrams 
of the form 
\begin{equation*}
\labellist
\pinlabel $\lambda$ at  -64 74  
\pinlabel $\dotsc$  at  134 63   
\tiny \hair 2pt
\pinlabel $\llambda_{j_1}$  at  -12 66
\pinlabel $j_1$  at   2 10  \pinlabel $j_2$   at  46 10
\pinlabel $j_3$  at  90 10  \pinlabel $j_k$   at 180 10
\endlabellist
\figins{-19}{0.70}{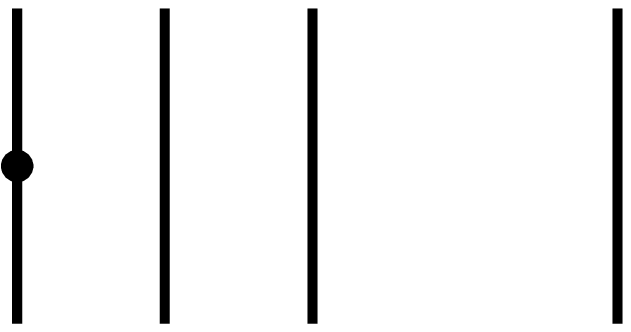} 
\end{equation*}
where the leftmost strand has $\llambda_{j_1}$ dots on it. 
We always label the leftmost region of a diagram with the partition $\lambda$ 
to indicate it is in $R_{\g}^\lambda$. 
The following was proved in~\cite{W1}.
\begin{lem}\label{lem:frobKLR}
The cyclotomic KLR algebra $R_{\g}^\llambda$ is Frobenius. 
\end{lem}

\smallskip

Projective modules over $R_{\g}^\lambda$ are defined the same way as for $R_{\g}$. 
Denote by $R^{\lambda}_{\g}\amod$ and by $R^{\lambda}_{\g}\prmod$ the categories of finitely generated, graded 
$R^{\lambda}_{\g}$-modules and finitely generated, graded projective $R^{\lambda}_{\g}$-modules
respectively.  
The module category structure in the $R^{\lambda}_{\g}\amod$ was studied in much more detail in \emph{e.g.}~\cite{KK} 
or~\cite{W1}. 
Here we describe the necessary to proceed through this paper.  
Let $\imath_i\colon R_{\g}^{\lambda}(\nu)\to R_{\g}^{\lambda}(\nu+\alpha_i)$ be the map obtained by adding 
a vertical strand labeled $i$ on the right of a diagram from $R^{\lambda}_{\g}$. 
The categorical $\g$-action on $R_{\g}^{\lambda}$ is obtained by 
the pair of biadjoint exact functors defined by 
\begin{equation}\label{eq:Uaction}
\begin{split}
F_i^{\lambda}   & 
\colon\mspace{22mu} R^{\lambda}_{\g}(\nu)\amod\mspace{16mu} 
\xra{\mspace{63mu}\Ind_i\mspace{63mu}}  R^{\lambda}_{\g}(\nu+\alpha_i)\amod 
\\
E_i^{\lambda}  
&\colon R^{\lambda}_{\g}(\nu+\alpha_i)\amod \xra{\ \Res_i\{\alpha_i^\vee(\lambda-\nu)-1\}\ }  
\mspace{20mu}R^{\lambda}_{\g} (\nu)\amod  
\end{split}
\end{equation}

The following is a particular case of a more general result proved by Kang and Kashiwara in~\cite{KK} 
and independently by Webster in~\cite{W1} . 
\begin{thm}[Kang-Kashiwara, Webster]\label{thm:KKW}
There is an isomorphism of $\g$-representations
\begin{equation*}
V_{\lambda}^{\g} \cong K_0\bigl( R_{\g}^{\lambda}\bigr) .
\end{equation*}
Moreover 
\begin{equation*}
\gdim\Hom_{R^\lambda_{\g}\amod}(P,P') = \langle [P],[P'] \rangle ,
\end{equation*}
where $\langle~,~\rangle$ is the $q$-Shapovalov form.  
\end{thm}

This isomorphism sends the isomorphism class of an indecomposable ${}_{j_1\dotsm j_r}P^\lambda$ to the 
 weight vector $F_{j_r}\dotsm F_{j_1}v_{\lambda}$. 
\smallskip

For two cyclotomic KLR algebras $R^\lambda$ and $R^{\lambda'}$ 
with $\lambda$ and $\lambda'$ highest weights we have the obvious surjection 
\begin{equation*}
\varphi_{\lambda'}^{\lambda}\colon R_{\g}^{\lambda}\to R_{\g}^{\lambda'}
\end{equation*}
which is compatible from the projections $R_{\g}\to R_{\g}^{\lambda}$ and $R_{\g}\to R_{\g}^{\lambda'}$. 
We have the functors
\begin{equation*}
\ext_{\lambda'}^\lambda: R_{\g}^{\lambda}\amod\to R_{\g}^{\lambda'}
\mspace{45mu}\text{and}\mspace{45mu}
\res_{\lambda'}^\lambda: R_{\g}^{\lambda'}\amod\to R_{\g}^{\lambda}
\end{equation*} 
of extension of scalars and restriction of scalars by $\varphi_{\lambda'}^\lambda$. 
These functors do not commute with the categorical $\g$-action in~\eqref{eq:Uaction} 
but they define maps corresponding with the inclusion/projection between 
the underlying vector spaces $K_0(R_{\g}^{\lambda})\rightleftarrows K_0(R_{\g}^{\lambda'})$. 

Since $\varphi_{\lambda'}^\lambda$ is compatible with the map $\imath_i$ of adding a strand labeled 
$i$ on the right 
we have isomorphisms $\ext_{\lambda'}^\lambda F_i^\lambda\cong F_i^{\lambda'}\ext_{\lambda'}^{\lambda}$ 
and $E_i^{\lambda}\res_{\lambda'}^{\lambda}\cong \res_{\lambda'}^\lambda E_i^{\lambda'}$.  
Note that $\ext_{\lambda'}^\lambda E_i^\lambda\ncong E_i^{\lambda'}\ext_{\lambda'}^{\lambda}$ 
and $F_i^{\lambda}\res_{\lambda'}^{\lambda}\ncong \res_{\lambda'}^\lambda F_i^{\lambda'}$. 
 
\begin{lem}\label{lem:cycproj}
In the case of 
$\lambda' = (\lambda_1-c,\lambda_2,\dotsc ,\lambda_n)$  
we have isomorphisms 
$\ext_{\lambda'}^\lambda E_i^\lambda\cong E_i^{\lambda'}\ext_{\lambda'}^{\lambda}$ 
and $F_i^{\lambda}\res_{\lambda'}^{\lambda}\cong \res_{\lambda'}^\lambda F_i^{\lambda'}$ 
for $i\neq 1$, and the maps $K_0(\ext_{\lambda'}^\lambda)$ and $K_0(\res_{\lambda'}^\lambda)$ 
are maps of $\g_{n-1}$-representations. 
\end{lem}

\section{Categorical branching rules}\label{sec:catBR}

In this section we obtain the categorical version of the branching rule. 
The approach follows closely the one in~\cite{vaz2} for the case of type $A_n$. 
From now on $\g_n$ denotes $\soo$, $\spp$ or $\so$. 
The first result is the following. 
\begin{thm}\label{thm:cycproj}
We have a surjection of cyclotomic KLR algebras 
\begin{equation*}
\pi^\lambda\colon R^\lambda_{\g_n} \to
\bigoplus\limits_{(\mu,\nu)\in\tau_{\g_n}(\lambda)}R_{\g_{n-1}}^{\xi_{(\mu,\nu)}(\lambda)} .
\end{equation*} 
\end{thm}

We stress 
using a $\xi(\lambda)$ with $\xi$ in $\cD_{\g_n,\lambda}$ is the same as using 
a $\xi_{(\mu,\nu)}$ with $(\mu,\nu)$ in $\tau_{\g_n}(\lambda)$ 
(see Definitions~\ref{def:admB},~\ref{def:admD} and~\ref{def:admC} and the remarks following them). 
We will use this fact implicitly in the sequel.

Using the map above we can define the functors of interest to us in this paper.
\begin{defn}
Let
\begin{align*}
\Pi^\lambda= \ext^\lambda
\colon
R^\lambda_{\g_n} \amod
&\to\   
\bigoplus\limits_{(\mu,\nu)\in\tau_{\g_n}(\lambda)}R^{\xi_{(\mu,\nu)}(\lambda)}_{\g_{n-1}}\amod 
\\[1ex]
 M &\mapsto M \otimes_{R^\lambda_{\g_n}}
\Bigl(\oplus_{ (\mu,\nu)\in\tau_{\g_n}(\lambda) } R^{\xi_{(\mu,\nu)}(\lambda)}_{\g_{n-1}}\Bigr)
\end{align*}
and $\res^\lambda\colon\bigl(\oplus_{ (\mu,\nu)\in\tau_{\g_n}(\lambda) }R^{\xi_{(\mu,\nu)}(\lambda)}_{\g_{n-1}}\bigr)\amod 
\to 
R^\lambda_{B_n}(k\alpha_1)\amod$ 
be the functors of extension of scalars and restriction of scalars 
by the map $\pi^\lambda$ from Theorem~\ref{thm:cycproj} respectively. 
\end{defn}

\begin{prop}\label{prop:intertwiner}
The functors $\Pi^\lambda$ and $\res^\lambda$ are biadjoint. 
The functor $\Pi^\lambda$ is full, essentially surjective and intertwines the categorical $\g_{n-1}$-action.
\end{prop}

The main result of this section is the folowing.
\begin{thm}\label{thm:branchrules}
The functor $\Pi^{\lambda}$ descends to 
an isomorphism of $\g_{n-1}$-representations
\begin{align*} 
K_0(\Pi^{\lambda} )
\colon V^{\g_n}_{\lambda} \cong K_0 ( R^\lambda_{\g_n} )
\xra{\ \ \cong\ \ }
K_0\biggl(\ \bigoplus\limits_{ (\mu,\nu)\in\tau_{\g_n}(\lambda) }R^{\xi_{(\mu,\nu)}(\lambda)}_{\g_{n-1}}\biggr)
\cong\bigoplus\limits_{\mu}c(\mu)V_{\mu}^{\mathfrak{so}_{2n-1}} ,
\end{align*}  
with $c(\mu)$ the number of $\nu=(\nu_2, \dotsc , \nu_n)$ such that the pair $(\nu,\mu)$ is in  $\tau_{B_n}(\lambda)$.  
\end{thm}


Recall that a functor is called a \emph{wide equivalence of categories}  
if it is full and a bijection on objects.
By analogy with the case of type $A_n$ we make the following conjecture. 
\begin{conj}\label{cor:BRinj}
The functor $\Pi^{\lambda}$ is injective on objects and therefore a wide equivalence of categories.  
\end{conj}

For the sake of readability we separate the proofs of Theorems~\ref{thm:cycproj} and~\ref{thm:branchrules} 
and Proposition~\ref{prop:intertwiner} in the cases of $\soo$, $\soo$ and $\spp$, 
which we do in the next three subsections. 
These can be read independently from each other.

\subsection{The case of quantum $\soo$}\label{ssec:catbranchingB}

We start by defining the special classes of idempotents in $R_{B_n}^\lambda$.
\begin{defn}
We define the idempotents 
$p_{\pm i}$ by
\begin{align*}
p_{-i} &= 
\begin{cases} 
p_{i, i+1, \dotsc , n-1, n, n, n-1, \dotsc, 2, 1}  
\\[0.5ex] 
p_{n, n-1, n-2, \dotsc, 2, 1}    
\end{cases} 
& \begin{matrix}  1 < i \leq n, \\[1.2ex] i=1, \end{matrix}
\\[1ex]
\ \ p_{+i} &= p_{i, i-1, i-2, \dotsc, 2, 1}    & 1 \leq i < n.
\end{align*}
\end{defn}
We denote the horizontal composition $p_{\pm j}p_{\pm k}$ by $p_{\pm j,\pm k}$. 
We allow ourselves another notational simplification and write 
$p_{\varepsilon j, \varepsilon k}=p_{\varepsilon jk}$, whenever $\varepsilon = \pm 1 $.

\begin{defn}
The idempotent 
$e'(p_{-\und{t}^\ell,+\und{i}^m},\und{j})\in R_{B_n}^\lambda(\nu+(\ell+m)\alpha_1)$  
is said to be a \emph{special idempotent} if 
$\xi_{-\und{t}^\ell,+\und{i}^m}$ is $\lambda$-admissible.  
We use the notation $e(p_{-\und{t}^\ell,\und{i}^m},\und{j})$ for special idempotents.  
\end{defn}

\medskip

In the following we give the maps between some cyclotomic KLR algebras that are necessary 
to obtain the categorical branching rule.  
\begin{lem}\label{lem:cycincB} 
For each $k \geq 0$ there is a surjection of algebras 
\begin{equation*}
\pi_{k}\colon R^\lambda_{B_n}(k\alpha_1) \to
\bigoplus\limits_{\xi_{-\und{t}^\ell,+\und{i}^m}\in\cD_\lambda^{k}}R_{B_{n-1}}^{\xi_{-\und{t}^\ell,+\und{i}^m}(\lambda)}.
\end{equation*} 
\end{lem}

\begin{proof}
We first prove that for each each 
$\xi_{-\und{t}^\ell,+\und{i}^m}$ $\lambda$-admissible   
we have a surjection of algebras 
\begin{equation*}
\pi_{-\und{t}^\ell,+\und{i}^m}\colon 
R^\lambda_{B_n}((m+\ell)\alpha_1) \to
R_{B_{n-1}}^{\xi_{-\und{t}^\ell,+\und{i}^m}(\lambda)} .
\end{equation*} 
To this end it is enough to show that for each 
$\xi_{-\und{t}^\ell,+\und{i}^m}$ as above  
the subalgebra 
\begin{align*} 
B_{-\und{t}^\ell,+\und{i}^m} = &
\bigoplus\limits_{\und{r},\und{s}\in\seq\{\alpha_2,\dotsc, \alpha_n\}}
e(p_{-\und{t}^\ell,+\und{i}^m},\und{r})
\bigl(R^\lambda_{B_n}((m+\ell)\alpha_1)\bigr)
e(p_{-\und{t}^\ell,+\und{i}^m},\und{s})
\end{align*} 
projects onto  $R_{B_{n-1}}^{\xi_{-\und{t}^\ell,+\und{i}^m}(\lambda)}$.

Let 
$\widetilde{B}_{-\und{t}^\ell,+\und{i}^m}\subset {B}_{-\und{t}^\ell,+\und{i}^m}$ 
be the subalgebra generated by 
all elements having a representative given by diagrams 
which consist of a horizontal composition 
of $p_{-\und{t}^\ell,+\und{i}^m}$ and a diagram from 
$R^{\lambda}_{B_n}(0.\alpha_1)$ 
from left to right, as below
\begin{equation*}
\labellist
\small 
\pinlabel $\lambda$ at -30 95 
\pinlabel $p_{-\und{t}^\ell}$ at 70 88
\pinlabel $p_{+\und{i}^m}$ at 236 88 
\pinlabel $R^\lambda_{B_n}(0.\alpha_1)$ at 412 90  
\pinlabel $\dotsc$ at 70 145   \pinlabel $\dotsc$ at 415 145 
\pinlabel $\dotsc$ at 70  35   \pinlabel $\dotsc$ at 415  35 
\pinlabel $\dotsc$ at 235 145 
\pinlabel $\dotsc$ at 235  35  
\endlabellist 
\figins{-20.5}{1.1}{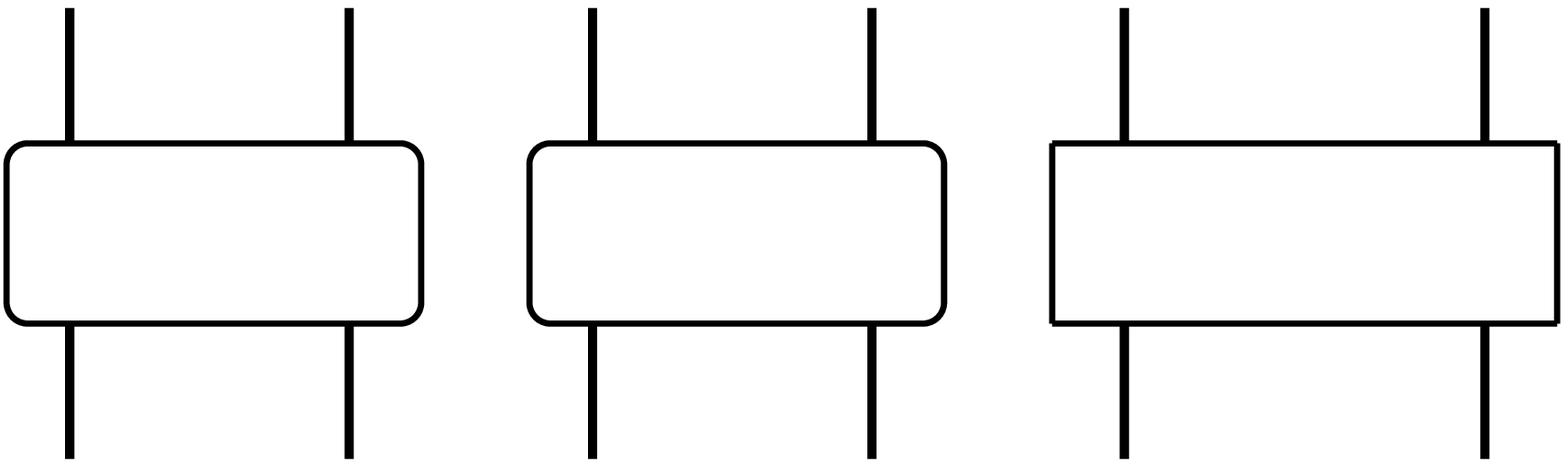}
\end{equation*}
Let $\widetilde{B}_{-\und{t}^\ell,+\und{i}^m}^\bot$ be its complement vector space. 
Moreover, let also $\widetilde{B}_{-\und{t}^\ell,+\und{i}^m}^\zeta$ be the quotient of 
$\widetilde{B}_{-\und{t}^\ell,+\und{i}^m}$ 
by the two sided ideal generated by all diagrams of the form 
\begin{equation*}
\labellist
\small 
\pinlabel $\lambda$ at -30 95 
\pinlabel $p_{+\und{i}^m}$ at 236 88 
\pinlabel $p_{-\und{t}^\ell}$ at 68 88
\pinlabel $\dotsc$ at 70 145   \pinlabel $\dotsc$ at 412 90  
\pinlabel $\dotsc$ at 70  35    
\pinlabel $\dotsc$ at 235 145 
\pinlabel $\dotsc$ at 235  35  
\tiny \hair 2pt
\pinlabel $j_1$ at 356  5 
\pinlabel $j_{\ell}$ at 468  6 
\pinlabel $\zeta_{j_1}$ at 335  96  
\endlabellist 
\figins{-20.5}{1.1}{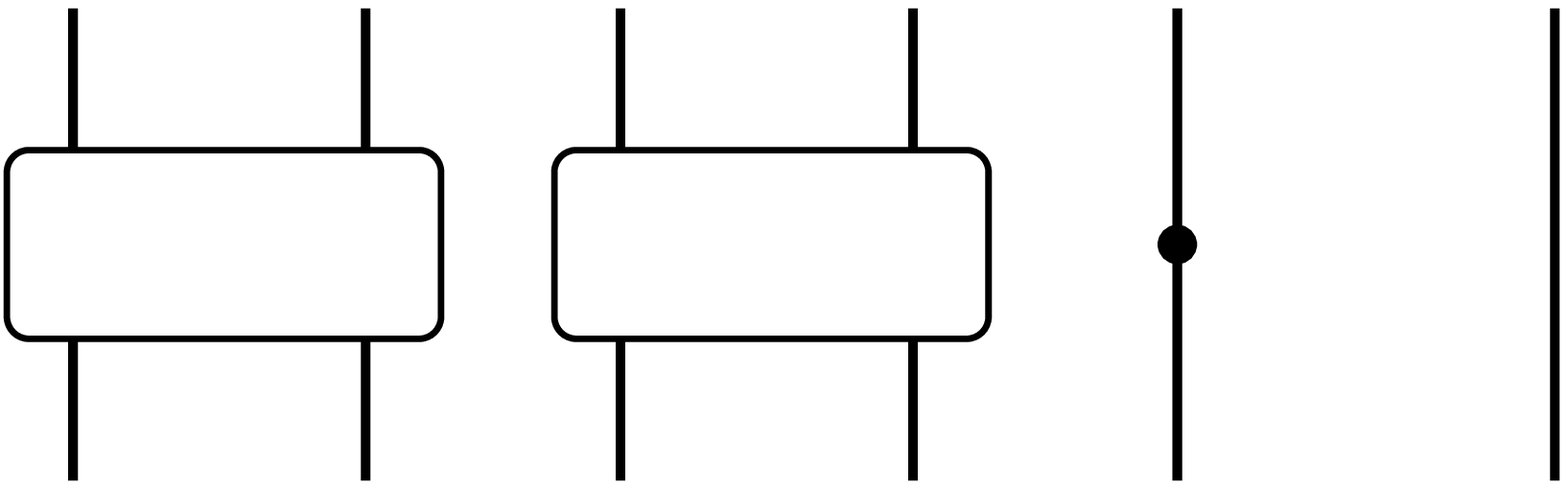}
\end{equation*}
where $\zeta=\overline{\xi_{-\und{t}^\ell,+\und{i}^m}(\lambda)}$. 
The algebras $\widetilde{B}_{-\und{t}^\ell,+\und{i}^m}^\zeta$ and $R_{B_n}^{\xi_{-\und{t}^\ell,+\und{i}^m}(\lambda)}$ 
are isomorphic by the map that sends $p_{-\und{t}^\ell,+\und{i}^m}X\in\widetilde{B}_{-\und{t}^\ell,+\und{i}^m}^\zeta$
to $X\in R_{B_n}^{\xi_{-\und{t}^\ell,+\und{i}^m}(\lambda)}$.

It is enough to show the cases of $(\ell,m)=(1,0)$ and $(\ell,m)=(0,1)$, 
since the general case follows easily by recursion.
For this purpose we consider $r=(r_1,\dotsc ,r_{n-1})$ and compute
\begin{equation*}
X_{\pm i}(j,r_j) = \mspace{18mu}
\labellist
\pinlabel $p_{\pm i}$ at   102 63    
\pinlabel $\dotsc$ at  252 60 
\pinlabel $\dotsc$ at  104 17
\pinlabel $\dotsc$ at  104 108
\pinlabel $\lambda$ at -15 95 
\tiny \hair 2pt
\pinlabel $r_j$    at  -6   65
\pinlabel $j$      at 148   -1
\endlabellist
\figins{-30}{0.90}{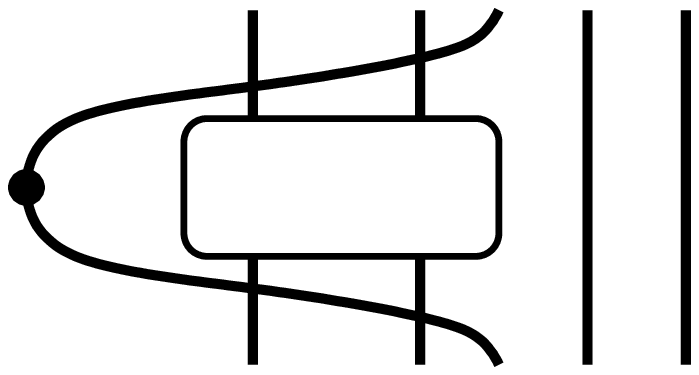}
\mspace{24mu}
\end{equation*}
where in the case of $+i$  we consider $1\leq i < n$ and in the case of 
$-i$ we take $1 \leq i \leq n$. 
The computation only uses the KLR relations and goes exactly as in~\cite{vaz2}. 
Denote by 
\begin{equation*}
Z_{\pm i}(j,\tilde{r}_j) = \mspace{18mu}
\labellist
\pinlabel $p_{\pm i}$ at  52 63    
\pinlabel $\dotsc$ at  230 60  
\pinlabel $\dotsc$ at   54 17
\pinlabel $\dotsc$ at   54 108
\pinlabel $\lambda$ at -15 95 
\tiny \hair 2pt
\pinlabel $\tilde{r}_j$    at  120   65
\pinlabel $j$      at 133   -1
\endlabellist
\figins{-30}{0.90}{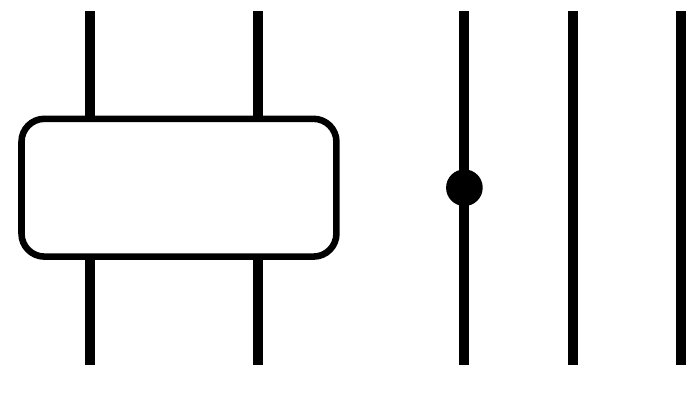}
\end{equation*}
For $X_{-i}(j,r_j)$ we obtain for $i \leq n-1$
\begin{align*}
X_{-i}(j,r_j) &= 
\begin{cases}
Z_{-i}(j,r_j+1) + \text{ terms in }\widetilde{B}_{\pm i}^\bot \mspace{35mu} & j=i-1,
\\[0,5ex] 
Z_{-i}(j,r_j-1) + \text{ terms in }\widetilde{B}_{\pm i}^\bot & j = i,
\\[0,5ex] 
Z_{-i}(j,r_j) + \text{ terms in }\widetilde{B}_{\pm i}^\bot & \text{else},
\end{cases}
\end{align*}
while for $i=n$ we get
\begin{align*}
X_{-n}(j,r_j) &= 
\begin{cases}
Z_{-n}(j,r_j+1) + \text{ terms in }\widetilde{B}_{\pm i}^\bot \mspace{35mu} & j=n-1,
\\[0,5ex] 
Z_{-n}(j,r_j-2) + \text{ terms in }\widetilde{B}_{\pm i}^\bot & j = n,
\\[0,5ex] 
Z_{-n}(j,r_j) + \text{ terms in }\widetilde{B}_{\pm i}^\bot & \text{else}. 
\end{cases}
\end{align*}
For $X_{+i}(j,r_j)$ we obtain for $i < n-1$
\begin{align*}
X_{+i}(j,r_j) &= 
\begin{cases}
Z_{+j}(j,r_j-1) + \text{ terms in }\widetilde{B}_{\pm i}^\bot \mspace{35mu} & j=i,
\\[0,5ex] 
Z_{+j}(j,r_j+1) + \text{ terms in }\widetilde{B}_{\pm i}^\bot & j = i+1,
\\[0,5ex] 
Z_{+j}(j,r_j) + \text{ terms in }\widetilde{B}_{\pm i}^\bot & \text{else}. 
\end{cases}
\end{align*}
while for $i=n-1$ we get 
\begin{align*}
X_{+(n-1)}(j,r_j) &= 
\begin{cases}
Z_{+(n-1)}(j,r_j-1) + \text{ terms in }\widetilde{B}_{\pm i}^\bot \mspace{35mu} & j = n-1,
\\[0,5ex] 
Z_{+(n-1)}(j,r_j+2) + \text{ terms in }\widetilde{B}_{\pm i}^\bot & j = n,
\\[0,5ex] 
Z_{+(n-1)}(j,r_j) + \text{ terms in }\widetilde{B}_{\pm i}^\bot & \text{else}. 
\end{cases}
\end{align*}
(we leave the details to the reader). 
Altogether we get that 
\begin{equation*}
\labellist
\pinlabel $p_{\pm i}$ at   102 63    
\pinlabel $\dotsc$ at  233 60 
\pinlabel $\dotsc$ at  104 17
\pinlabel $\dotsc$ at  104 108
\pinlabel $\lambda$ at -15 95 
\tiny \hair 2pt
\pinlabel $r_j$    at  -6   65
\pinlabel $j$      at 148   -1
\endlabellist
\figins{-30}{0.90}{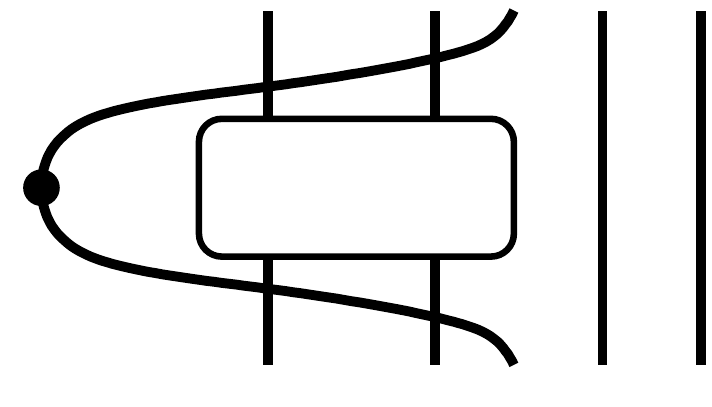}
\mspace{54mu} = \mspace{24mu}
\labellist
\pinlabel $p_{\pm i}$ at  52 63    
\pinlabel $\dotsc$ at  230 60  
\pinlabel $\dotsc$ at   54 17
\pinlabel $\dotsc$ at   54 108
\pinlabel $\lambda$ at -15 95 
\tiny \hair 2pt
\pinlabel $\overline{\xi_{\pm i}(r)}_j$   at  105   95
\pinlabel $j$      at 133   -1
\endlabellist
\figins{-30}{0.90}{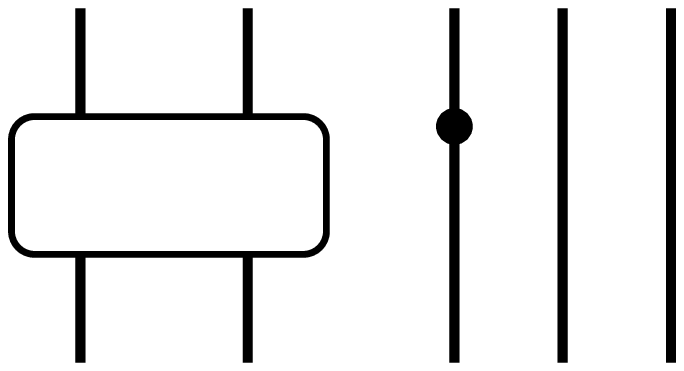}
\mspace{42mu}
+\text{ terms in }\widetilde{B}_{\pm i}^\bot 
\end{equation*}
Taking $r_i=\overline{\lambda}_i$ we see that $X_{\pm i}(j,\overline{\lambda}_j)$ consists 
of a sum of a term in 
$\widetilde{B}_{\pm i}^\zeta\cong R_{B_{n-1}}^{\xi_{\pm i}(\lambda)}$ 
with terms in $\widetilde{B}_{\pm i}^\bot$. 
This shows that 
$R_{B_n}^\lambda(\alpha_1)$ projects onto $R_{B_{n-1}}^{\xi_{\pm i}(\lambda)}$. 
We call this projection $\pi_{\pm i}$.  The kernel of $\pi_{\pm i}$ is the two-sided ideal 
generated by the elements in $\widetilde{B}_{\pm i}^\bot$ involved above.  
Proceeding recursively one gets that $\pi_{-\und{t}^\ell,+\und{i}^m}$ is a surjection of algebras. 
The lemma now follows from the observation that $R_{B_n}^\lambda((m+\ell)\alpha_1)$ projects canonically onto
$\bigoplus_{\xi_{-\und{t}^\ell,+\und{i}^m}\in\cD_\lambda^{m+\ell}}{B}_{-\und{t}^\ell,+\und{i}^m}$.  
\end{proof}
Summing over $k$ in Lemma~\ref{lem:cycincB} we have the following.
\begin{cor}\label{cor:cycincB}
We have a surjection of algebras 
\begin{equation*}
\pi^\lambda\colon 
R^\lambda_{B_n} \to
\bigoplus\limits_{\xi_{-\und{t}^\ell,+\und{i}^m}\in\cD_{B_n,\lambda}}R_{B_{n-1}}^{ \xi_{-\und{t}^\ell,+\und{i}^m}(\lambda) } .
\end{equation*} 
\end{cor}
This proves the case of $\g_n=\soo$ in Theorem~\ref{thm:cycproj}. 
Moreover, by an adequate choice of the maps $\pi_{-\und{t}^\ell,+\und{i}^m}$, we can rephrase Lemma~\ref{lem:cycincB} 
and Corollary~\ref{cor:cycincB}. 
\begin{cor}  \label{cor:cycinccB} 
For each $(\nu,\mu)\in\tau_{B_n}(\lambda)$ there is a surjection of algebras 
\begin{equation*}
\pi_{(\nu,\mu)}\colon
R^\lambda_{B_n} \to
R_{B_{n-1}}^{\mu }   
\end{equation*} 
giving rise to a surjection of algebras
\begin{equation*}
\pi\colon
R^\lambda_{B_n} \to
\bigoplus\limits_{(\nu,\mu)\in\,\tau_{B_n}(\lambda)}R_{B_{n-1}}^{\mu} .
\end{equation*} 
\end{cor}

\medskip

Fix a $k\geq 1$ and let 
\begin{align*}
\Pi_k^\lambda= \ext_k^\lambda
\colon
R^\lambda_{B_n}(k\alpha_1)\amod
&\to\   
\bigoplus\limits_{\xi_{-\und{t}^\ell,+\und{i}^m}\in\cD_\lambda^{k}}R^{\xi_{-\und{t}^\ell,+\und{i}^m}(\lambda)}_{B_{n-1}}\amod 
\\[1ex]
 M &\mapsto M \otimes_{R^\lambda_{B_n}(k\alpha_1)}
\Bigl(\oplus_{\xi_{-\und{t}^\ell,+\und{i}^m}}R^{\xi_{-\und{t}^\ell,+\und{i}^m}(\lambda)}_{B_{n-1}}\Bigr)
\end{align*}
and $\res_k^\lambda\colon\bigl(\oplus_{\xi_{-\und{t}^\ell,+\und{i}^m}}R^{\xi_{-\und{t}^\ell,+\und{i}^m}(\lambda)}_{B_{n-1}}\bigr)\amod 
\to 
R^\lambda_{B_n}(k\alpha_1)\amod$ 
be respectively the functors of extension of scalars and restriction of scalars 
by the map $\pi_k$ from Lemma~\ref{lem:cycincD}.

Using the surjections 
$\pi_{-\und{t}^\ell,+\und{i}^m}\colon R_{B_n}^{\lambda}(k\alpha_1)\to R_{B_{n-1}}^{\xi_{-\und{t}^\ell,+\und{i}^m}(\lambda)}$ 
for each $\xi_{-\und{t}^\ell,+\und{i}^m}$ that are inherited from the map $\pi_k$  we call 
\begin{align*}
\Pi_{-\und{t}^\ell,+\und{i}^m}^\lambda\colon R_{B_n}^{\lambda}(k\alpha_1)\amod &\to R_{B_{n-1}}^{\xi_{-\und{t}^\ell,+\und{i}^m}(\lambda)}\amod 
\intertext{and} 
\res_{-\und{t}^\ell,+\und{i}^m}^\lambda\colon R_{B_{n-1}}^{\xi_{-\und{t}^\ell,+\und{i}^m}(\lambda)}\amod &\to R_{B_n}^{\lambda}(k\alpha_1)\amod
\end{align*} 
the \emph{components} of $\Pi_k^\lambda$ and $\res_k^\lambda$.

%

The following can be proved in the same way as in~\cite{vaz2} for the case of $\sln$.

\begin{lem}\label{lem:BR-biadj-BRFull-branch-k-B}
The functors $\Pi_k^\lambda$ and $\res_k^\lambda$ are biadjoint.
The functor $\Pi_k^\lambda$ is full and essentially surjective
and each $\Pi^\lambda_{-\und{t}^\ell,+\und{i}^m}$  
intertwines the categorical $\mathfrak{so}_{2n-1}$-action. 
\end{lem}

\medskip

Finally define the functor 
\begin{align*}
\Pi^{\lambda}=\bigoplus_{k\geq 0}\Pi^\lambda_{k} 
&\colon 
 R^\lambda_{B_n}\amod
\to
\bigoplus\limits_{\xi_{-\und{t}^\ell,+\und{i}^m}\in\cD_\lambda}R^{\xi_{-\und{t}^\ell,+\und{i}^m}(\lambda)}_{B_{n-1}}\amod .
\end{align*} 
Summing over $k$ in Lemma~\ref{lem:BR-biadj-BRFull-branch-k-B} 
proves  the case of $\g_n=\soo$ in Proposition~\ref{prop:intertwiner}.

\n Each component $\Pi_{-\und{t}^\ell,+\und{i}^m}$ of $\Pi_k^\lambda$ descends to  
a surjection 
$$K_0(\Pi_{-\und{t}^\ell,+\und{i}^m})\colon K_0(R_{B_n}^\lambda((m+\ell)\alpha_1))\to K_0(R_{B_{n-1}}^{\xi_{-\und{t}^\ell+\und{i}^m}(\lambda)})$$ 
between the respective Grothendieck groups. 
This surjection intertwines 
the  $\mathfrak{so}_{2n-1}$-action by Lem\-ma~\ref{lem:BR-biadj-BRFull-branch-k-B}.  
The main result of this section now follows by counting dimensions.
\begin{thm}\label{thm:branchrulesB}
Functor $\Pi^{\lambda}$ descends to 
an isomorphism of $\mathfrak{so}_{2n-1}$-representations
\begin{align*} 
K_0(\Pi^{\lambda} )
\colon V^{\soo}_{\lambda} \cong K_0 ( R^\lambda_{B_n} )
\xra{\ \ \cong\ \ }
K_0\biggl(\ \bigoplus\limits_{\xi_{-\und{t}^\ell,+\und{i}^m}\in\cD_\lambda}R^{\xi_{-\und{t}^\ell,+\und{i}^m}(\lambda)}_{B_{n-1}}\biggr)
\cong\bigoplus\limits_{\mu}c(\mu)V_{\mu}^{\mathfrak{so}_{2n-1}} ,
\end{align*}  
with $c(\mu)$ the number of $\nu=(\nu_2, \dotsc , \nu_n)$ such that the pair $(\nu,\mu)$ is in  $\tau_{B_n}(\lambda)$.  
\end{thm}
This proves the case of $\g_n=\soo$ in Theorem~\ref{thm:branchrules}.

\subsection{The case of quantum $\so$}\label{ssec:catbranchingD}

We start by defining the special classes of idempotents in $R_{D_n}^\lambda$.

\begin{defn}
We define the idempotents 
$p_{\pm i}$ by
\begin{align*}
p_{-i} &= p_{i, i+1, \dotsc, n-2, n-1, n, n-2, n-3, \dotsc, 2, 1}  &  1 < i \leq n, 
\\
p_{+i} &= p_{i, i-1, i-2, \dotsc, 2, 1}  &  1 \leq i < n. 
\end{align*}
\end{defn}
We denote the horizontal composition $p_{\pm j}p_{\pm k}$ by $p_{\pm j,\pm k}$. 
We allow ourselves another notational simplification and write 
$p_{\varepsilon j, \varepsilon k}=p_{\varepsilon jk}$, whenever $\varepsilon = \pm 1 $.

\begin{defn}
The idempotent 
$e'(p_{+\und{i}^m,-\und{t}^\ell},\und{j})\in R_{D_n}^\lambda(\nu+(m+\ell)\alpha_1)$  
is said to be a \emph{special idempotent} if 
$\xi_{+\und{i}^m,-\und{t}^\ell}$ is $\lambda$-admissible.  
We use the notation $e(p_{+\und{i}^m,-\und{t}^\ell},\und{j})$ for special idempotents.  
\end{defn}

\medskip

In the following we give the maps between some cyclotomic 
KLR algebras that are necessary to obtain the categorical branching rule.

\begin{lem}\label{lem:cycincD} 
For each $k \geq 0$ there is a surjection of algebras 
\begin{equation*}
\pi_{k}\colon
R^\lambda_{D_n}(k\alpha_1) \to
\bigoplus\limits_{\xi_{+\und{i}^m,-\und{t}^\ell}\in\cD_\lambda^{k}}R_{D_{n-1}}^{\xi_{+\und{i}^m,-\und{t}^\ell}(\lambda)}.
\end{equation*} 
\end{lem}

\begin{proof} 
We first prove that for each each 
$\xi_{+\und{i}^m,-\und{t}^\ell}$ $\lambda$-admissible   
we have a surjection of algebras 
\begin{equation*}
\pi_{+\und{i}^m,-\und{t}^\ell}\colon
R^\lambda_{D_n}((m+\ell)\alpha_1) \to
R_{D_{n-1}}^{\xi_{+\und{i}^m,-\und{t}^\ell}(\lambda)} .
\end{equation*} 
To this end it is enough to show that for each 
$\xi_{+\und{i}^m,-\und{t}^\ell}$ as above the subalgebra 
\begin{align*} 
D_{+\und{i}^m,-\und{t}^\ell} = &
\bigoplus\limits_{\und{r},\und{s}\in\seq\{\alpha_2,\dotsc ,\alpha_n\}}
e(p_{+\und{i}^m,-\und{t}^\ell},\und{r})
\bigl(R^\lambda_{D_n}((m+\ell)\alpha_1)\bigr)
e(p_{+\und{i}^m,-\und{t}^\ell},\und{s})
\end{align*} 
projects onto  $R_{D_{n-1}}^{\xi_{+\und{i}^m,-\und{t}^\ell}(\lambda)}$.

Let 
$\widetilde{D}_{+\und{i}^m,-\und{t}^\ell}\subset {D}_{+\und{i}^m,-\und{t}^\ell}$ 
be the subalgebra generated by 
all elements having a representative given by diagrams 
which consist of a horizontal composition 
of $p_{+\und{i}^m,-\und{t}^\ell}$ and a diagram from $R^{\lambda}_{D_n}(0.\alpha_1)$
from left to right, as below
\begin{equation*}
\labellist
\small 
\pinlabel $\lambda$ at -30 95 
\pinlabel $p_{+\und{i}^m}$ at 72 88 
\pinlabel $p_{-\und{t}^\ell}$ at 234 88
\pinlabel $R^\lambda_{D_n}(0.\alpha_1)$ at 412 90  
\pinlabel $\dotsc$ at 70 145   \pinlabel $\dotsc$ at 415 145 
\pinlabel $\dotsc$ at 70  35   \pinlabel $\dotsc$ at 415  35 
\pinlabel $\dotsc$ at 235 145 
\pinlabel $\dotsc$ at 235  35  
\endlabellist 
\figins{-20.5}{1.1}{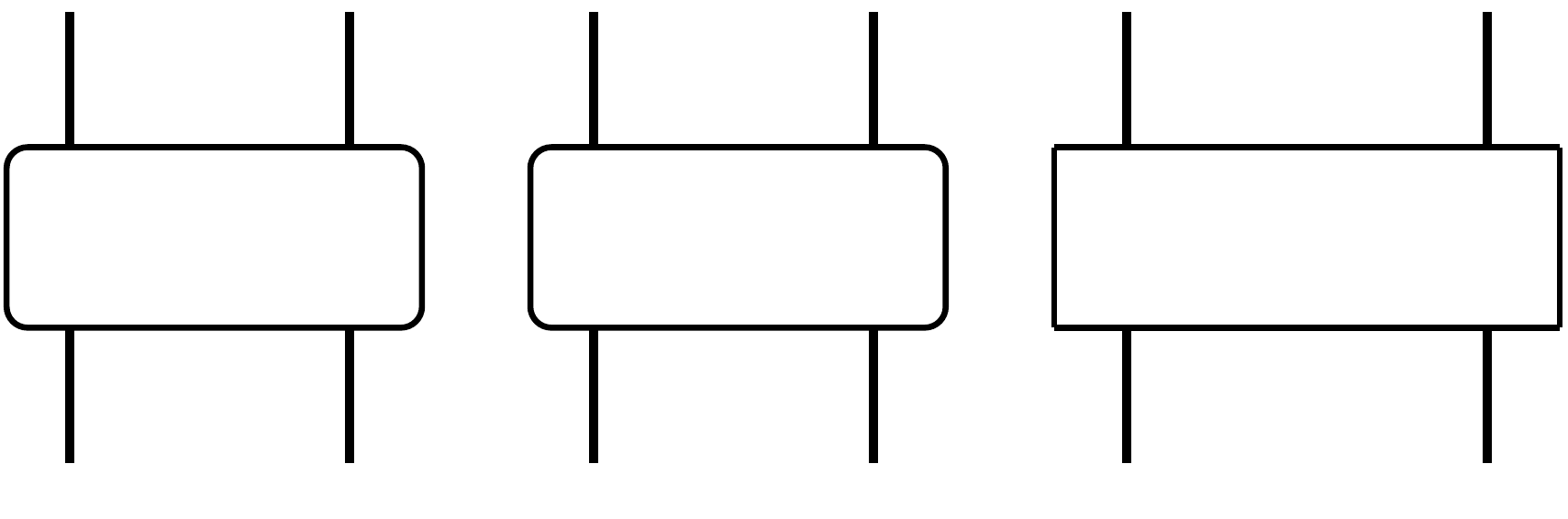}
\end{equation*}
Let $\widetilde{D}_{+\und{i}^m,-\und{t}^\ell}^\bot$ be its complement vector space. 
Moreover, let also $\widetilde{D}_{+\und{i}^m,-\und{t}^\ell}^\zeta$ be the quotient of $\widetilde{D}_{+\und{i}^m,-\und{t}^\ell}$ 
by the two sided ideal 
generated by all diagrams of the form 
\begin{equation*}
\labellist
\small 
\pinlabel $\lambda$ at -30 95 
\pinlabel $p_{+\und{i}^m}$ at 72 88 
\pinlabel $p_{-\und{t}^\ell}$ at 234 88
\pinlabel $\dotsc$ at 70 145   \pinlabel $\dotsc$ at 412 90  
\pinlabel $\dotsc$ at 70  35    
\pinlabel $\dotsc$ at 235 145 
\pinlabel $\dotsc$ at 235  35  
\tiny \hair 2pt
\pinlabel $j_1$ at 356  5 
\pinlabel $j_{\ell}$ at 468  6 
\pinlabel $\zeta_{j_1}$ at 335  96  
\endlabellist 
\figins{-20.5}{1.1}{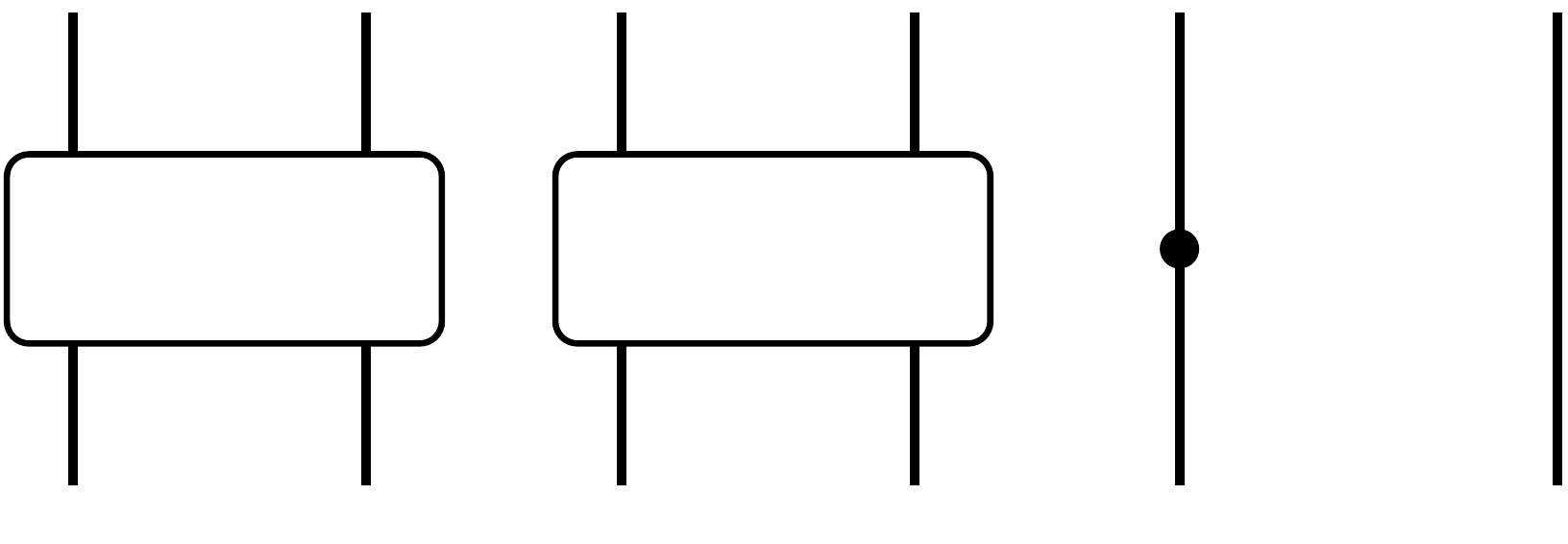}
\end{equation*}
where $\zeta=\overline{\xi_{+\und{i}^m,-\und{t}^\ell}(\lambda)}$. 
The algebras $\widetilde{D}_{+\und{i}^m,\und{t}^\ell}^\zeta$ and $R_{D_n}^{\xi_{+\und{i}^m,-\und{t}^\ell}(\lambda)}$ 
are isomorphic by the map that sends $p_{+\und{i}^m,-\und{t}^\ell}X$ to $X$.

It is enough to show the cases of $(m,\ell)=(1,0)$ and $(m,\ell)=(0,1)$, 
sincethe general case follows easily by recursion.
For this purpose we consider $r=(r_1,\dotsc ,r_{n-1})$ and compute
\begin{equation*}
X_{\pm i}(j,r_j) = \mspace{18mu}
\labellist
\pinlabel $p_{\pm i}$ at   102 63    
\pinlabel $\dotsc$ at  252 63 
\pinlabel $\dotsc$ at  104 17
\pinlabel $\dotsc$ at  104 108
\pinlabel $\lambda$ at -15 95 
\tiny \hair 2pt
\pinlabel $r_j$    at  -6   65
\pinlabel $j$      at 148   -1
\endlabellist
\figins{-30}{0.90}{pcycR2}
\mspace{24mu}
\end{equation*}
where in the case of $+i$  we consider $1\leq i < n$ and in the case of 
$-i$ we take $1 < i \leq n$. 
The computation only uses the KLR relations and goes exactly as in~\cite{vaz2}. 

Denote by 
\begin{equation*}
Z_{\pm i}(j,\tilde{r}_j) = \mspace{18mu}
\labellist
\pinlabel $p_{\pm i}$ at  52 63    
\pinlabel $\dotsc$ at  230 60  
\pinlabel $\dotsc$ at   54 17
\pinlabel $\dotsc$ at   54 108
\pinlabel $\lambda$ at -15 95 
\tiny \hair 2pt
\pinlabel $\tilde{r}_j$    at  120   65
\pinlabel $j$      at 133   -1
\endlabellist
\figins{-30}{0.90}{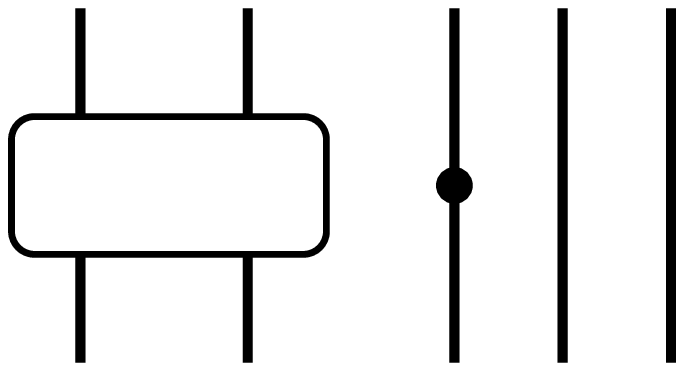}
\end{equation*}
For the $X_{-i}(j,r_j)$ we obtain for $i \leq n-1$
\begin{align*}
X_{-i}(j,r_j) &= 
\begin{cases}
Z_{-i}(j,r_j+1) + \text{ terms in }\widetilde{D}_{\pm i}^\bot \mspace{35mu} & j=i-1,
\\[0,5ex] 
Z_{-i}(j,r_j-1) + \text{ terms in }\widetilde{D}_{\pm i}^\bot & j = i,
\\[0,5ex] 
Z_{-i}(j,r_j) + \text{ terms in }\widetilde{D}_{\pm i}^\bot & \text{else}. 
\end{cases}
\end{align*}
For $X_{+i}(j,r_j)$ we obtain
\begin{align*}
X_{+i}(j,r_j) &= 
\begin{cases}
Z_{+i}(j,r_j-1) + \text{ terms in }\widetilde{D}_{\pm i}^\bot \mspace{35mu} & j=i,
\\[0,5ex] 
Z_{+i}(j,r_j+1) + \text{ terms in }\widetilde{D}_{\pm i}^\bot & j = i+1 \neq n, 
\\[0,5ex] 
Z_{+i}(j,r_j+1) + \text{ terms in }\widetilde{D}_{\pm i}^\bot & i=n-2, j = n, 
\\[0,5ex] 
Z_{+i}(j,r_j+1) + \text{ terms in }\widetilde{D}_{\pm i}^\bot & i=n-1, j = n, 
\\[0,5ex] 
Z_{+i}(j,r_j) + \text{ terms in }\widetilde{D}_{\pm i}^\bot & \text{else} 
\end{cases}
\end{align*}
(we leave the details to the reader). 
Altogether we get that 
\begin{equation*}
\labellist
\pinlabel $p_{\pm i}$ at   102 63    
\pinlabel $\dotsc$ at  233 60 
\pinlabel $\dotsc$ at  104 17
\pinlabel $\dotsc$ at  104 108
\pinlabel $\lambda$ at -15 95 
\tiny \hair 2pt
\pinlabel $r_j$    at  -6   65
\pinlabel $j$      at 148   -1
\endlabellist
\figins{-30}{0.90}{pcycR2}
\mspace{54mu} = \mspace{24mu}
\labellist
\pinlabel $p_{\pm i}$ at  52 63    
\pinlabel $\dotsc$ at  230 60  
\pinlabel $\dotsc$ at   54 17
\pinlabel $\dotsc$ at   54 108
\pinlabel $\lambda$ at -15 95 
\tiny \hair 2pt
\pinlabel $\overline{\xi_{\pm i}(r)}_j$   at  105   95
\pinlabel $j$      at 133   -1
\endlabellist
\figins{-30}{0.90}{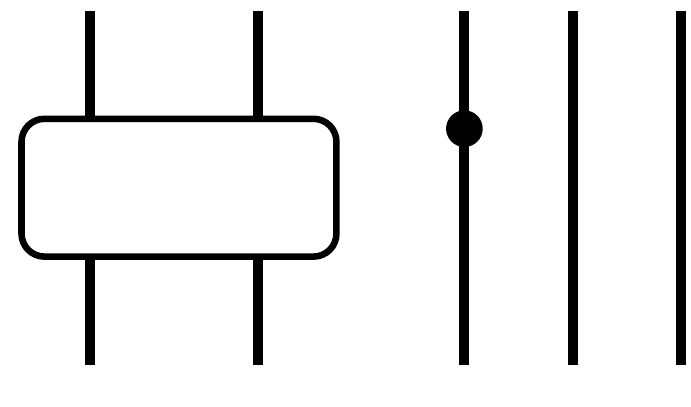}
\mspace{42mu}
+\text{ terms in }\widetilde{D}_{\pm i}^\bot 
\end{equation*}
Taking $r_i=\overline{\lambda}_i$ we see that $X_{\pm i}(j,\overline{\lambda}_j)$ consists 
of a sum of a term in 
$\widetilde{D}_{\pm i}^\zeta\cong R_{D_{n-1}}^{\xi_{\pm i}(\lambda)}$ 
with terms in $\widetilde{D}_{\pm i}^\bot$. 
This shows that 
$R_{D_n}^\lambda(\alpha_1)$ projects onto $R_{D_{n-1}}^{\xi_{\pm i}(\lambda)}$. 
We call this projection $\pi_{\pm i}$.  The kernel of $\pi_{\pm i}$ is the two-sided ideal 
generated by the elements in $\widetilde{D}_{\pm i}^\bot$ involved above.  
Proceeding recursively one gets that $\pi_{+\und{i}^m,-\und{t}^\ell}$ is a surjection of algebras. 
The lemma now follows from the observation that $R_{D_n}^\lambda((m+\ell)\alpha_1)$ projects canonically onto
$\bigoplus_{\xi_{+\und{i}^m,-\und{t}^\ell}\in\cD_\lambda^{m+\ell}}{D}_{+\und{i}^m,-\und{t}^\ell}$.  
\end{proof}
Summing over $k$ in Lemma~\ref{lem:cycincD} we have the following.
\begin{cor}\label{cor:cycincD}
We have a surjection of algebras 
\begin{equation*}
\pi^\lambda\colon
R^\lambda_{D_n} \to
\bigoplus\limits_{\xi_{+\und{i}^m,-\und{t}^\ell}\in\cD_{D_n,\lambda}}R_{D_{n-1}}^{ \xi_{+\und{i}^m,-\und{t}^\ell}(\lambda) } .
\end{equation*} 
\end{cor}
This proves Theorem~\ref{thm:cycproj} in the case of $\g_n=\so$. 
Moreover, by an adequate choice of the maps $\pi_{+\und{i}^m,-\und{t}^\ell}$, we can rephrase Lemma~\ref{lem:cycincD} 
and Corollary~\ref{cor:cycincD}. 
\begin{cor}  \label{cor:cycinccD} 
For each $(\nu,\mu)\in\tau_{D_n}(\lambda)$ there is a surjection of algebras 
\begin{equation*}
\pi_{(\nu,\mu)}\colon
R^\lambda_{D_n} \to
R_{D_{n-1}}^{\mu }   
\end{equation*} 
giving rise to a surjection of algebras
\begin{equation*}
\pi\colon
R^\lambda_{D_n} \to
\bigoplus\limits_{(\nu,\mu)\in\,\tau_{D_n}(\lambda)}R_{D_{n-1}}^{\mu} .
\end{equation*} 
\end{cor}

\medskip

Fix a $k\geq 1$ and let 
\begin{align*}
\Pi_k^\lambda= \ext_k^\lambda
\colon
R^\lambda_{D_n}(k\alpha_1)\amod
&\to\   
\bigoplus\limits_{\xi_{+\und{i}^m,-\und{t}^\ell}\in\cD_\lambda^{k}}R^{\xi_{+\und{i}^m,\und{t}^\ell}(\lambda)}_{D_{n-1}}\amod 
\\[1ex]
 M &\mapsto M \otimes_{R^\lambda_{D_n}(k\alpha_1)}
\Bigl(\oplus_{\xi_{+\und{i}^m,-\und{t}^\ell}}R^{\xi_{+\und{i}^m,-\und{t}^\ell}(\lambda)}_{D_{n-1}}\Bigr)
\end{align*}
and $\res_k^\lambda\colon\bigl(\oplus_{\xi_{+\und{i}^m,-\und{t}^\ell}}R^{\xi_{+\und{i}^m,\und{t}^\ell}(\lambda)}_{D_{n-1}}\bigr)\amod 
\to 
R^\lambda_{D_n}(k\alpha_1)\amod$ 
be respectively the functors of extension of scalars and restriction of scalars 
by the map $\pi_k$ from Lemma~\ref{lem:cycincD}.

Using the surjections 
$\pi_{+\und{i}^m,-\und{t}^\ell}\colon R_{D_n}^{\lambda}(k\alpha_1)\to R_{D_{n-1}}^{\xi_{+\und{i}^m,-\und{t}^m}(\lambda)}$
for each $\xi_{+\und{i}^m,-\und{t}^\ell}$ that are inherited from the map $\pi_k$  we call 
\begin{align*}
\Pi_{+\und{i}^m,-\und{t}^\ell}^\lambda\colon R_{D_n}^{\lambda}(k\alpha_1)\amod &\to R_{D_{n-1}}^{\xi_{+\und{i}^m,-\und{t}^\ell}(\lambda)}\amod 
\intertext{and} 
\res_{+\und{i}^m,-\und{t}^\ell}^\lambda\colon R_{D_{n-1}}^{\xi_{+\und{i}^m,-\und{t}^\ell}(\lambda)}\amod &\to R_{D_n}^{\lambda}(k\alpha_1)\amod
\end{align*} 
the \emph{components} of $\Pi_k^\lambda$ and $\res_k^\lambda$.

The following can be proved in the same way as in~\cite{vaz2} for the case of $\sln$.

\begin{lem}\label{lem:BR-biadj-BRFull-branch-k-D}
The functors $\Pi_k^\lambda$ and $\res_k^\lambda$ are biadjoint.
The functor $\Pi_k^\lambda$ is full and essentially surjective
and each $\Pi^\lambda_{+\und{i}^m,-\und{t}^\ell}$  
intertwines the categorical $\mathfrak{so}_{2n-2}$-action. 
\end{lem}

\medskip

Finally define the functor 
\begin{align*}
\Pi^{\lambda}=\bigoplus_{k\geq 0}\Pi^\lambda_{k} 
&\colon 
 R^\lambda_{D_n}\amod
\to
\bigoplus\limits_{\xi_{+\und{i}^m,-\und{t}^\ell}\in\cD_\lambda}R^{\xi_{+\und{i}^m,-\und{t}^\ell}(\lambda)}_{D_{n-1}}\amod .
\end{align*} 
Summing over $k$ in Lemma~\ref{lem:BR-biadj-BRFull-branch-k-D} 
proves the  case of $\g_n=\so$ in Proposition~\ref{prop:intertwiner}. 
\n Each component $\Pi_{+\und{i}^m,-\und{t}^\ell}$ of $\Pi_k^\lambda$ descends to  
a surjection 
$$K_0(\Pi_{+\und{i}^m,-\und{t}^\ell})\colon K_0(R_{D_n}^\lambda((m+\ell)\alpha_1))\to K_0(R_{D_{n-1}}^{\xi_{+\und{i}^m,-\und{t}^\ell}(\lambda)})$$ 
between the respective Grothendieck groups. 
This surjection intertwines 
the  $\mathfrak{so}_{2n-2}$-action by Lem\-ma~\ref{lem:BR-biadj-BRFull-branch-k-D}.  
The main result of this section now follows by counting dimensions.
\begin{thm}\label{thm:branchrulesD}
Functor $\Pi^{\lambda}$ descends to 
an isomorphism of $\mathfrak{so}_{2n-2}$-representations
\begin{align*} 
K_0(\Pi^{\lambda} )
\colon V^{\so}_{\lambda} \cong K_0 ( R^\lambda_{D_n} )
\xra{\ \ \cong\ \ }
K_0\biggl(\ \bigoplus\limits_{\xi_{+\und{i}^m,-\und{t}^\ell}\in\cD_\lambda}R^{\xi_{+\und{i}^m,-\und{t}^\ell}(\lambda)}_{D_{n-1}}\biggr)
\cong\bigoplus\limits_{\mu}c(\mu)V_{\mu}^{\mathfrak{so}_{2n-2}} ,
\end{align*}  
with $c(\mu)$ the number of $\nu=(\nu_2, \dotsc , \nu_n)$ such that the pair $(\nu,\mu)$ is in  $\tau_{D_n}(\lambda)$.  
\end{thm}
This proves the case of $\g_n=\so$ in Theorem~\ref{thm:branchrules}.

\subsection{The case of quantum $\spp$}\label{ssec:catbranchingC}

Due to the particular form of $\nu$ in the Gelfand-Tsetlin pattern~\eqref{eq:GTpattC} 
for type $C_n$
and the fact that in the case of $\spp$ no positive root $\delta$ satisfies 
$\lambda -\delta = (\lambda_1-1,\lambda_2,\dotsc, \lambda_n)$ 
we need to use a different technique. 
We stress that, although satisfying $\nu_1 \geq \dotsm \geq \nu_n$, the $n$-tuple $\nu$  
is not to be interpreted as a $\spp$-weighest weight here. 
To obtain the categorical branching rule in this case we proceed in two steps. 
As before we start by defining the special classes of idempotents in $R_{C_n}^\lambda$. 
\begin{defn}
We define the idempotents 
$p_{\pm i}$ by
\begin{align*}
p_{-i} &= 
\begin{cases} 
p_{i, i+1, \dotsc , n-1,  n, n-1, \dotsc, 2, 1}  
\\[0.5ex] 
p_{n, n-1, n-2, \dotsc, 2, 1}  
\end{cases} 
& \begin{matrix}  1 < i \leq n, \\[1.2ex] i=n, \end{matrix}
\\[1ex]
\ \ p_{+i} &= p_{i, i-1, i-2, \dotsc, 2, 1}    & 1 \leq i < n. 
\end{align*}
\end{defn}
Notice there is no $p_{-1}$. 
We denote the horizontal composition $p_{\pm j}p_{\pm k}$ by $p_{\pm j,\pm k}$. 
As before and write 
$p_{\varepsilon j, \varepsilon k}=p_{\varepsilon jk}$, whenever $\varepsilon = \pm 1 $.

\begin{defn}
The idempotent 
$e'(p_{-\und{t}^\ell,+\und{i}^m},\und{j})\in R_{B_n}^\lambda(\nu+(\ell+m)\alpha_1)$  
is said to be a \emph{special idempotent} if 
for fixed $c$, the map $\xi_{-\und{t}^\ell,+\und{i}^m}$ is $\lambda c$-admissible.  
We use the notation $e(p_{-\und{t}^\ell,\und{i}^m},\und{j})$ for special idempotents.  
\end{defn}

\medskip

In the following we give the maps between the cyclotomic KLR algebras 
that are necessary to obtain the categorical branching rule.  
We simplify notation and write $\lambda-c = (\lambda_1-c,\lambda_2,\dotsc \lambda_n)$. 
\begin{lem}\label{lem:cycincC} 
For each $c=0,\dotsc ,\lambda_1-\lambda_2$ and $k \geq 0$ there is a surjection of algebras 
\begin{equation*}
\pi_{k}\colon
R^{\lambda-c}_{C_n}(k\alpha_1) \to
\bigoplus\limits_{\xi_{-\und{t}^\ell,+\und{i}^m}\in\cD_{\lambda,c}^{k}}R_{C_{n-1}}^{\xi_{-\und{t}^\ell,+\und{i}^m}(\lambda)}.
\end{equation*} 
\end{lem}

\begin{proof}
We first prove that for each each 
$\xi_{-\und{t}^\ell,+\und{i}^m}$ $\lambda c$-admissible   
we have a surjection of algebras 
\begin{equation*}
\pi_{-\und{t}^\ell,+\und{i}^m}\colon
R^{\lambda-c}_{C_n}((\ell+m)\alpha_1) \to
R_{C_{n-1}}^{\xi_{-\und{t}^\ell,+\und{i}^m}(\lambda)} .
\end{equation*} 
To this end it is enough to show that for each 
$\xi_{-\und{t}^\ell,+\und{i}^m}$ as above, 
the subalgebra 
\begin{align*} 
C_{-\und{t}^\ell,+\und{i}^m} = &
\bigoplus\limits_{\und{r},\und{s}\in\seq\{\alpha_2,\dotsc, \alpha_n\}}
e(p_{-\und{t}^\ell,+\und{i}^m},\und{r})
\bigl(R^{\lambda-c}_{C_n}((\ell+m)\alpha_1)\bigr)
e(p_{-\und{t}^\ell,+\und{i}^m},\und{s})
\end{align*} 
projects onto  $R_{C_{n-1}}^{\xi_{-\und{t}^\ell,+\und{i}^m}(\lambda)}$.

Let 
$\widetilde{C}_{-\und{t}^\ell,+\und{i}^m}\subset {C}_{-\und{t}^\ell,+\und{i}^m}$ 
be the subalgebra generated by 
all elements having a representative given by diagrams 
which consist of a horizontal composition 
of $p_{-\und{t}^\ell,+\und{i}^m}$ and a diagram from 
$R^{\lambda-c}_{C_n}(0.\alpha_1)$ 
from left to right, as below
\begin{equation*}
\labellist
\small 
\pinlabel $\lambda$ at -30 95 
\pinlabel $p_{-\und{t}^\ell}$ at 70 88
\pinlabel $p_{+\und{i}^m}$ at 236 88 
\pinlabel $R^{\lambda-c}_{C_n}(0.\alpha_1)$ at 412 90  
\pinlabel $\dotsc$ at 70 145   \pinlabel $\dotsc$ at 415 145 
\pinlabel $\dotsc$ at 70  35   \pinlabel $\dotsc$ at 415  35 
\pinlabel $\dotsc$ at 235 145 
\pinlabel $\dotsc$ at 235  35  
\endlabellist 
\figins{-20.5}{1.1}{ppbox-cycq}
\end{equation*}
Let $\widetilde{C}_{-\und{t}^\ell,+\und{i}^m}^\bot$ be its complement vector space. 
Moreover, let also $\widetilde{C}_{-\und{t}^\ell,+\und{i}^m}^\zeta$ be the quotient of 
$\widetilde{C}_{-\und{t}^\ell,+\und{i}^m}$ 
by the two sided ideal generated by all diagrams of the form 
\begin{equation*}
\labellist
\small 
\pinlabel $\lambda$ at -30 95 
\pinlabel $p_{+\und{i}^m}$ at 236 88 
\pinlabel $p_{-\und{t}^\ell}$ at 68 88
\pinlabel $\dotsc$ at 70 145   \pinlabel $\dotsc$ at 412 90  
\pinlabel $\dotsc$ at 70  35    
\pinlabel $\dotsc$ at 235 145 
\pinlabel $\dotsc$ at 235  35  
\tiny \hair 2pt
\pinlabel $j_1$ at 356  5 
\pinlabel $j_{\ell}$ at 468  6 
\pinlabel $\zeta_{j_1}$ at 335  96  
\endlabellist 
\figins{-20.5}{1.1}{ppA-cycl}
\end{equation*}
where $\zeta=\overline{\xi_{-\und{t}^\ell,+\und{i}^m}(\lambda)}$. 
The algebras $\widetilde{C}_{-\und{t}^\ell,+\und{i}^m}^\zeta$ and $R_{C_n}^{\xi_{-\und{t}^\ell,+\und{i}^m}(\lambda)}$ 
are isomorphic by the map that sends $p_{-\und{t}^\ell,+\und{i}^m}X$ to $X$.

As before it is enough to show the cases of $(\ell,m)=(1,0)$ and $(\ell,m)=(0,1)$, 
since the general case follows easily by recursion.
For this purpose we consider $r=(r_1,\dotsc ,r_{n-1})$ and compute
\begin{equation*}
X_{\pm i}(j,r_j) = \mspace{18mu}
\labellist
\pinlabel $p_{\pm i}$ at   102 63    
\pinlabel $\dotsc$ at  252 60 
\pinlabel $\dotsc$ at  104 17
\pinlabel $\dotsc$ at  104 108
\pinlabel $\lambda$ at -15 95 
\tiny \hair 2pt
\pinlabel $r_j$    at  -6   65
\pinlabel $j$      at 148   -1
\endlabellist
\figins{-30}{0.90}{pcycR2}
\mspace{24mu}
\end{equation*}
where in the case of $+i$  we consider $1\leq i < n$ and in the case of 
$-i$ we take $1 \leq i \leq n$. 
The computation only uses the KLR relations and goes exactly as in~\cite{vaz2}. 
Denote by  
\begin{equation*}
Z_{\pm i}(j,\tilde{r}_j) = \mspace{18mu}
\labellist
\pinlabel $p_{\pm i}$ at  52 63    
\pinlabel $\dotsc$ at  230 60  
\pinlabel $\dotsc$ at   54 17
\pinlabel $\dotsc$ at   54 108
\pinlabel $\lambda$ at -15 95 
\tiny \hair 2pt
\pinlabel $\tilde{r}_j$    at  120   65
\pinlabel $j$      at 133   -1
\endlabellist
\figins{-30}{0.90}{pcycR2j}
\end{equation*}
For $X_{-i}(j,r_j)$ we obtain for $i \leq n-1$
\begin{align*}
X_{-i}(j,r_j) &= 
\begin{cases}
Z_{-j}(j,r_j+1) + \text{ terms in }\widetilde{C}_{\pm i}^\bot \mspace{35mu} & j=i-1,
\\[0,5ex] 
Z_{+j}(j,r_j-1) + \text{ terms in }\widetilde{C}_{\pm i}^\bot & j = i, 
\\[0,5ex] 
Z_{+j}(j,r_j) + \text{ terms in }\widetilde{C}_{\pm i}^\bot & \text{else}, 
\end{cases}
\end{align*}
while for $i=n$ we get 
\begin{align*}
X_{-n}(j,r_j) &= 
\begin{cases}
Z_{-n}(j,r_j+1) + \text{ terms in }\widetilde{C}_{\pm i}^\bot \mspace{35mu} & j = n-1, 
\\[0,5ex] 
Z_{-n}(j,r_j-1) + \text{ terms in }\widetilde{C}_{\pm i}^\bot & j = n, 
\\[0,5ex] 
Z_{-n}(j,r_j) + \text{ terms in }\widetilde{C}_{\pm i}^\bot & \text{else}. 
\end{cases}
\end{align*}
For $X_{+i}(j,r_j)$ we obtain for $1 \leq i < n$
\begin{align*}
X_{+i}(j,r_j) &= 
\begin{cases}
Z_{+i}(j,r_j-1) + \text{ terms in }\widetilde{C}_{\pm i}^\bot & j = i, 
\\[0,5ex] 
Z_{+i}(j,r_j+1) + \text{ terms in }\widetilde{C}_{\pm i}^\bot \mspace{35mu} & j=i+1, 
\\[0,5ex] 
Z_{+i}(j,r_j) + \text{ terms in }\widetilde{C}_{\pm i}^\bot & \text{else}
\end{cases}
\end{align*}
(we leave the details to the reader).

Altogether we get that 
\begin{equation*}
\labellist
\pinlabel $p_{\pm i}$ at   102 63    
\pinlabel $\dotsc$ at  233 60 
\pinlabel $\dotsc$ at  104 17
\pinlabel $\dotsc$ at  104 108
\pinlabel $\lambda$ at -15 95 
\tiny \hair 2pt
\pinlabel $r_j$    at  -6   65
\pinlabel $j$      at 148   -1
\endlabellist
\figins{-30}{0.90}{pcycR2}
\mspace{54mu} = \mspace{24mu}
\labellist
\pinlabel $p_{\pm i}$ at  52 63    
\pinlabel $\dotsc$ at  230 60  
\pinlabel $\dotsc$ at   54 17
\pinlabel $\dotsc$ at   54 108
\pinlabel $\lambda$ at -15 95 
\tiny \hair 2pt
\pinlabel $\overline{\xi_{\pm i}(r)}_j$   at  105   95
\pinlabel $j$      at 133   -1
\endlabellist
\figins{-30}{0.90}{pcycR2jj}
\mspace{42mu}
+\text{ terms in }\widetilde{C}_{\pm i}^\bot 
\end{equation*}
Taking $r_i=\overline{\lambda}_i$ we see that $X_{\pm i}(j,\overline{\lambda}_j)$ consists 
of a sum of a term in 
$\widetilde{C}_{\pm i}^\zeta\cong R_{B_{n-1}}^{\xi_{\pm i}(\lambda)}$ 
with terms in $\widetilde{C}_{\pm i}^\bot$. 
This shows that 
$R_{C_n}^{\lambda-c}(\alpha_1)$ projects onto $R_{C_{n-1}}^{\xi_{\pm i}(\lambda)}$. 
We call this projection $\pi_{\pm i}$.  The kernel of $\pi_{\pm i}$ is the two-sided ideal 
generated by the elements in $\widetilde{C}_{\pm i}^\bot$ involved above.  
Proceeding recursively one gets that $\pi_{-\und{t}^\ell,+\und{i}^m}$ is a surjection of algebras. 
The lemma now follows from the observation that $R_{C_n}^{\lambda-c}((m+\ell)\alpha_1)$ projects canonically onto
$\bigoplus_{\xi_{-\und{t}^\ell,+\und{i}^m}\in\cD_{\lambda,c}^{m+\ell}}{C}_{-\und{t}^\ell,+\und{i}^m}$.  
\end{proof}

Making use of the surjections $\varphi_{\lambda-c}^\lambda\colon R^{\lambda}_{C_n} \to R^{\lambda-c}_{C_n}$
from Subsection~\ref{ssec:cycKLR} 
and summing and over $k$ in Lemma~\ref{lem:cycincC} and over $c$ we have the following 
(the $\varphi$ are to be applied before the $\pi_{-\und{t}^\ell,+\und{i}^m}$s).
\begin{cor}\label{cor:cycincC}
We have a surjection of algebras 
\begin{equation*}
\pi^\lambda \colon
R^\lambda_{C_n} \to
\bigoplus\limits_{\xi_{-\und{t}^\ell,+\und{i}^m}\in\cD_{C_n,\lambda}}R_{C_{n-1}}^{ \xi_{-\und{t}^\ell,+\und{i}^m}(\lambda) } .
\end{equation*} 
\end{cor}
This proves the case of $\g_n=\spp$ in Theorem~\ref{thm:cycproj}. 
Moreover, by an adequate choice of the maps $\pi_{-\und{t}^\ell,+\und{i}^m}$, 
we can rephrase Lemma~\ref{lem:cycincC} 
and Corollary~\ref{cor:cycincC}. 
\begin{cor}  \label{cor:cycinccC} 
For each $(\nu,\mu)\in\tau_{B_n}(\lambda)$ there is a surjection of algebras 
\begin{equation*}
\pi_{(\nu,\mu)}\colon
R^\lambda_{C_n} \to
R_{C_{n-1}}^{\mu }   
\end{equation*} 
giving rise to a surjection of algebras
\begin{equation*}
R^\lambda_{C_n} \xra{\ \ \pi \ \ } \bigoplus\limits_{(\nu,\mu)\in\,\tau_{C_n}(\lambda)}R_{C_{n-1}}^{\mu} .
\end{equation*} 
\end{cor}

\medskip

Fix a $k\geq 1$ and let 
\begin{align*}
\Pi_k^{\lambda-c}= \ext_k^{\lambda-c}
\colon
R^{\lambda-c}_{C_n}(k\alpha_1)\amod
&\to\   
\bigoplus\limits_{\xi_{-\und{t}^\ell,+\und{i}^m}\in\cD_{\lambda,c}^{k}}R^{\xi_{-\und{t}^\ell,+\und{i}^m}(\lambda)}_{C_{n-1}}\amod 
\\[1ex]
 M &\mapsto M \otimes_{R^\lambda_{C_n}(k\alpha_1)}
\Bigl(\oplus_{\xi_{-\und{t}^\ell,+\und{i}^m}}R^{\xi_{-\und{t}^\ell,+\und{i}^m}(\lambda)}_{B_{n-1}}\Bigr)
\end{align*}
and $\res_k^{\lambda-c}\colon\bigl(\oplus_{\xi_{-\und{t}^\ell,+\und{i}^m}}R^{\xi_{-\und{t}^\ell,+\und{i}^m}(\lambda)}_{B_{n-1}}\bigr)\amod 
\to 
R^{\lambda-c}_{B_n}(k\alpha_1)\amod$ 
be respectively the functors of extension of scalars and restriction of scalars 
by the map $\pi_k$ from Lemma~\ref{lem:cycincC}.

Using the surjections 
$\pi_{-\und{t}^\ell,+\und{i}^m}\colon R_{C_n}^{\lambda-c}(k\alpha_1)\to R_{C_{n-1}}^{\xi_{-\und{t}^\ell,+\und{i}^m}(\lambda)}$ 
for each $\xi_{-\und{t}^\ell,+\und{i}^m}$ that are inherited from the map $\pi_k$  we call 
\begin{align*}
\Pi_{-\und{t}^\ell,+\und{i}^m}^{\lambda-c}\colon R_{C_n}^{\lambda}(k\alpha_1)\amod &\to R_{C_{n-1}}^{\xi_{-\und{t}^\ell,+\und{i}^m}(\lambda)}\amod 
\intertext{and} 
\res_{-\und{t}^\ell,+\und{i}^m}^{\lambda-c}\colon R_{B_{C-1}}^{\xi_{-\und{t}^\ell,+\und{i}^m}(\lambda)}\amod &\to R_{C_n}^{\lambda-c}(k\alpha_1)\amod
\end{align*} 
the \emph{components} of $\Pi_k^\lambda$ and $\res_k^\lambda$.
The following can be proved in the same way as in~\cite{vaz2} for the case of $\sln$.
\begin{lem}\label{lem:BR-biadj-BRFull-branch-k-C}
The functors $\Pi_k^{\lambda-c}$ and $\res_k^{\lambda-c}$ are biadjoint.
The functor $\Pi_k^{\lambda-c}$ is full and essentially surjective
and each $\Pi^{\lambda-c}_{-\und{t}^\ell,+\und{i}^m}$  
intertwines the categorical $\mathfrak{sp}_{2n-2}$-action. 
\end{lem}

\medskip

Define also the functor 
\begin{align*}
\Pi^{\lambda-c}=\bigoplus_{k\geq 0}\Pi^{\lambda-c}_{k} 
&\colon 
 R^{\lambda-c}_{C_n}\amod
\to
\bigoplus\limits_{\xi_{-\und{t}^\ell,+\und{i}^m}\in\cD_{\lambda,c}}R^{\xi_{-\und{t}^\ell,+\und{i}^m}(\lambda)}_{C_{n-1}}\amod .
\end{align*}

\n Each component $\Pi_{-\und{t}^\ell,+\und{i}^m}^{\lambda-c}$ of $\Pi_k^{\lambda-c}$ descends to  
a surjection 
$$K_0(\Pi_{-\und{t}^\ell,+\und{i}^m}^{\lambda-c})\colon K_0(R_{C_n}^{\lambda-c}((\ell+m)\alpha_1))\to K_0(R_{C_{n-1}}^{\xi_{-\und{t}^\ell+\und{i}^m}(\lambda)})$$ 
between the respective Grothendieck groups
for each $\xi_{-\und{t}^\ell+\und{i}^m}$ in $\cD_{\lambda,c}$. 
Precomposing with the maps $\varphi_{\lambda-c}^{\lambda}$ from Subsection~\ref{ssec:cycKLR}  
for $0\leq c \leq\lambda_1-\lambda_2$ we get surjections 
$$
K_0(\Pi_{-\und{t}^\ell,+\und{i}^m}^{\lambda-c}\ext_{\lambda-c}^\lambda) \colon 
K_0(R_{C_n}^{\lambda}((\ell+m)\alpha_1)) \to K_0(R_{C_n}^{\lambda-c}((\ell+m)\alpha_1))
\to K_0(R_{C_{n-1}}^{\xi_{-\und{t}^\ell+\und{i}^m}(\lambda)})
$$
which intertwine  
the  $\mathfrak{sp}_{2n-2}$-action by Lem\-ma~\ref{lem:BR-biadj-BRFull-branch-k-B} and the 
remarks at the end of Subsection~\ref{ssec:cycKLR}.  
Here 
$\ext_{\lambda-c}^\lambda\colon R_{C_n}^\lambda \to R_{C_n}^{\lambda - c}$
are the functors of extension of scalars by $\varphi_{\lambda-c}^{\lambda}$.

To use $\cD_\lambda$ instead of $\cD_{\lambda,c}$ we define for  $\xi_{-\und{t}^\ell+\und{i}^m}$ in $\cD_\lambda^{m + \ell}$
(see the comments at the end of Subsection~\ref{ssec:brulesC})
\begin{equation*}
\Pi_{-\und{t}^\ell,+\und{i}^m}^{\lambda}=\Pi_{-\und{t}^{\ell-c},+\und{i}^m}^{\lambda}\circ\ext_{\lambda-c}^\lambda
\colon R_{C_n}^{\lambda}((\ell+m)\alpha_1)\amod \to R_{C_{n-1}}^{\xi_{-\und{t}^\ell+\und{i}^m}(\lambda)}\amod .
\end{equation*} 
Finally we sum over $k=\ell+m$ to obtain  
\begin{equation*}
\Pi^{\lambda} = \bigoplus\limits_{\xi_{-\und{t}^\ell,+\und{i}^m}\in\cD_{\lambda}}\Pi_{-\und{t}^\ell,+\und{i}^m}^{\lambda}
\colon 
R_{C_n}^\lambda\amod \to 
\bigoplus\limits_{\xi_{-\und{t}^\ell,+\und{i}^m}\in\cD_{\lambda}}R_{C_{n-1}}^{\xi_{-\und{t}^\ell+\und{i}^m}(\lambda)}\amod .
\end{equation*}
This is more compact way of describing $\Pi^\lambda$ as the composite functor 
\begin{multline*}
\Pi^{\lambda} =\bigoplus_{c=0}^{\lambda_1-\lambda_2}\Pi^{\lambda-c} 
\colon
R^{\lambda}_{C_n}\amod
\xra{ \bigoplus\limits_{c=0}^{\lambda_1-\lambda_2} \ext_{\lambda-c}^\lambda}
\\ \bigoplus_{c=0}^{\lambda_1-\lambda_2} R^{\lambda-c}_{C_n}\amod \xra{\bigoplus\limits_{c=0}^{\lambda_1-\lambda_2}\Pi^{\lambda-c}} 
 \bigoplus_{c=0}^{\lambda_1-\lambda_2} \bigoplus\limits_{\xi_{-\und{t}^\ell,+\und{i}^m}\in\cD_{\lambda,c}}R^{\xi_{-\und{t}^\ell,+\und{i}^m}(\lambda)}_{C_{n-1}}\amod .
\end{multline*} 
Lemma~\ref{lem:BR-biadj-BRFull-branch-k-C} together Lemma~\ref{lem:cycproj} 
proves the case of $\g_n=\spp$ in Proposition~\ref{prop:intertwiner}. 
The main result of this section now follows by counting dimensions.
\begin{thm}\label{thm:branchrulesC}
Functor $\Pi^{\lambda}$ descends to 
an isomorphism of $\mathfrak{sp}_{2n-2}$-representations
\begin{align*} 
K_0(\Pi^{\lambda} )
\colon V^{\spp}_{\lambda} \cong K_0 ( R^\lambda_{C_n} )
\xra{\ \ \cong\ \ }
K_0\biggl(\ \bigoplus_{c=0}^{\lambda_1-\lambda_2}\bigoplus\limits_{\xi_{-\und{t}^\ell,+\und{i}^m}\in\cD_{\lambda,c}}R^{\xi_{-\und{t}^\ell,+\und{i}^m}(\lambda)}_{B_{C-1}}\biggr)
\cong\bigoplus\limits_{\mu}c(\mu)V_{\mu}^{\mathfrak{sp}_{2n-2}} ,
\end{align*}  
with $c(\mu)$ the number of $\nu=(\nu_1, \dotsc , \nu_n)$ such that the pair $(\nu,\mu)$ is in  $\tau_{C_n}(\lambda)$.  
\end{thm}
This proves the case of $\g_n=\spp$ in Theorem~\ref{thm:branchrules}.

\subsection{The cyclotomic quotient conjecture revisited}%
\label{ssec:easyBK}

We can no extend the types $B_n$, $C_n$ and $D_n$ the elementary proof of the 
Khovanov-Lauda cyclotomic conjecture given in~\cite{vaz2} for type $A_n$. 
The proof of the following is just an application {\it ipsis verbis} of the arguments 
of Theorem 5.19 in~\cite{vaz2}.
\begin{thm}
For $\g$ a quantum Kac-Moody algebra of classical type 
we have an isomorphism of $\g$-modules
$$K_0(R_{\g}^\lambda) \cong V_\lambda^{\g}.$$
\end{thm}

\section{The categorical Gelfand-Tsetlin basis}\label{sec:catGT}

Al the results of section 5 in~\cite{vaz2} remain true in the present case. 
To avoid repeting all the arguments we simply state the results and refer to~\cite{vaz2} 
for the proofs. 
To adjust the notation to the present paper we define functors 
\begin{equation*}
E_j^{\xi_{\mu,\nu}^{\mu',\nu'}(\lambda)}, F_j^{\xi_{\mu,\nu}^{\mu',\nu'}(\lambda)}
\colon
R_{n}^{\xi_{(\mu,\nu)}(\lambda)}\amod\to R_{n}^{\xi_{(\mu',\nu')}(\lambda)}\amod, 
\end{equation*}
for $(\mu,\nu)$, $(\mu',\nu')$ in $\tau_{\g_n}(\lambda)$ and $j=2,\dotsc , n$, by
\begin{align*}
E_j^{\xi_{\mu,\nu}^{\mu',\nu'}(\lambda)} &= \Pi^\lambda_{(\mu',\nu')}E_j^{\g_n}\res^\lambda_{(\mu,\nu)}
\intertext{and} 
F_j^{\xi_{\mu,\nu}^{\mu',\nu'}(\lambda)} &= \Pi^\lambda_{(\mu',\nu')}F_j^{\g_n}\res^\lambda_{(\mu,\nu)} .
\end{align*}
These functors are zero unless $(\mu,\nu)=(\mu',\nu')$.  
In this case they coincide with the functors $E_j^{\xi_{(\mu,\nu)}(\lambda)}$ and 
$F_j^{\xi_{(\mu,\nu)}(\lambda)}$ from the structure of categorical $\g_{n-1}$-module 
on $R_{\g_{n-1}}^{\xi_{(\mu,\nu)}(\lambda)}\amod$.

\begin{defn}
For $i\in\{2,\dotsc ,n\}$ we define the functors
\begin{align*}
F_j^{\g_{n-1}} = \bigoplus\limits_{(\mu,\nu)\in\tau_{\g_n}(\lambda)} F_j^{\xi_{\mu,\nu}^{\mu,\nu}(\lambda)}  
\mspace{25mu}\text{and}\mspace{25mu} 
E_j^{\g_{n-1}} = \bigoplus\limits_{(\mu,\nu)\in\tau_{\g_n}(\lambda)} E_j^{\xi_{\mu,\nu}^{\mu,\nu}(\lambda)}  
\end{align*}
with the obvious source and target categories.
\end{defn}

Note that, each of the subcategories $R_{\g_{n-1}}^{\xi_{(\mu,\nu)}(\lambda)}\prmod$ is stable under the action of 
the functors $F_j^{\g_{n-1}}$ and $E_j^{\g_{n-1}}$.

To treat the case of the functors
 $F_1^{\g_n}$ and $E_1^{\g_n}$ 
we proceed as in~\cite{vaz2}. 
Recall that for a projective $M$ in $R^{\xi_{(\mu,\nu)}\lambda}_{\g_{n-1}}\amod$ the module 
$\res^\lambda_{(\mu,\nu)}(M)$ in $R^{\lambda}_{\g_{n}}\amod$ 
which is not projective in general and therefore we cannot be sure that 
the modules 
$\Pi^\lambda_{(\mu',\nu')}F_1^{\g_n}\res^\lambda_{(\mu,\nu)}(M)$ 
and
$\Pi^\lambda_{(\mu',\nu')}E_1^{\g_n}\res^\lambda_{(\mu,\nu)}(M)$ 
are projective in $\bigoplus\limits_{(\mu,\nu)\in\tau_{\g_n}(\lambda)}R^{\xi_{(\mu,\nu)}\lambda}_{\g_{n-1}}\amod$.

To overcome this we define new functors 
\begin{equation*}
{F}_1^{\g_{n-1}},{E}_1^{\g_{n-1}}  
\colon  
\bigoplus\limits_{(\mu,\nu)\in\tau_{\g_n}(\lambda)}R^{\xi_{(\mu,\nu)}\lambda}_{\g_{n-1}}\amod\ 
\to
\bigoplus\limits_{(\mu,\nu)\in\tau_{\g_n}(\lambda)}R^{\xi_{(\mu,\nu)}\lambda}_{\g_{n-1}}\amod
\end{equation*}
by 
\begin{align*}
{F}_1^{\g_{n-1}}(M) = \bigoplus\limits_{(\mu,\nu)\in\tau_{\g_n}(\lambda)} 
 \Pi^\lambda_{(\mu',\nu')}F_1^{\g_n} P(\res^\lambda_{(\mu,\nu)}(M))
\intertext{and} 
{E}_1^{\g_{n-1}}(M) = \bigoplus\limits_{(\mu,\nu)\in\tau_{\g_n}(\lambda)} 
 \Pi^\lambda_{(\mu',\nu')}E_1^{\g_n} P(\res^\lambda_{(\mu,\nu)}(M))
\end{align*}
where $P(-)$ is the projective cover in $R^{\lambda}_{\g_{n}}\amod$.  

\begin{lem}
The pair of (biadjoint) endofunctors $\{ {F}_1^{\g_{n-1}}, {E}_1^{\g_{n-1}} \}$ 
take projectives to projectives and 
define a categorical $\mathfrak{sl}_2$-action on  
$\bigoplus\limits_{(\mu,\nu)\in \tau_{\g_n}(\lambda)}R^{\xi_{(\mu,\nu)}(\lambda)}_{\g_{n-1}}\prmod$. 
\end{lem}

\begin{lem}
The functors $\{F_j^{\g_{n-1}}, E_j^{\g_{n-1}}\}_{\in \{1,\dotsc ,n\}}$ define a categorical $\g_n$-action on 
$$\bigoplus\limits_{(\mu,\nu)\in\tau_{\g_n}(\lambda)}R^{\xi_{(\mu,\nu)}(\lambda)}_n\amod .$$  
\end{lem}

\begin{cor}
With the action $E_n^{\g_n}$ and $F_n^{\g_{n}}$ as above the isomorphism $K_0(\Pi^\lambda)$ in  
Theorem~\ref{thm:branchrules} is a surjection of $\g_n$-representations. 
\end{cor}

\subsection{Classes of special indecomposables and the Gelfand-Tsetlin basis}%

We now pass to des\-cribe the classes of special indecomposables corresponding to the Gelfand-Tsetlin basis. 
Let $s\in\cS(\lambda)$ be a complete Gelfand-Tsetlin pattern. 
Each such $s$ defines a special idempotent $e(s)$ obtained by horizontal composition of 
the special idempotents defined in Section~\ref{sec:catBR}, each one corresponding to 
a surjection $R_{\g_m}^{\eta}\to R_{\g_{m-1}}^{\eta'}$ as in Theorem~\ref{thm:cycproj}. 
We also define a functor
\begin{equation*}
\res^s \colon
\Bbbk\amod \to R_{\g_n}^\lambda\amod.
\end{equation*}
which is the composite of the restriction functors defined in Section~\ref{sec:catBR} 
corresponding to $s$.

It is not hard to see that the functor $\res^s$ takes the one-dimensional $\Bbbk$-module $\Bbbk$ to the module 
${}_{e(s)}L$ 
which is 
a quotient of the projective $R_{\g_n}^\lambda$-module 
${}_{e(s)}P$. 
This can be given a description in terms of diagrams as we did in type $A_n$.

As in~\cite{vaz2} we define an algebra $\widecheck{R}_{\g_n}^\lambda$ by 
\begin{equation}
\widecheck{R}_{\g_n}^\lambda =  
\bigoplus\limits_{s\in\cS(\lambda)} R_{\g_n}^\lambda/\ker(\pi_s). 
\end{equation}

We have $\Hom_{\widecheck{R}_{\g_n}^\lambda\amod}({}_{e(s)}L,{}_{e(s')}L)=0$ if $s\neq s'$. 
Moreover the module ${}_{e(s)}L$ is projective indecomposable 
as a module over $\widecheck{R}_{\g_n}^\lambda$.

We can continuing adjusting the construction in~\cite{vaz2} and 
define a pair of biadjoint endofunctors $\widecheck{F}_j^{\lambda}$ and $\widecheck{E}_j^{\lambda}$
on  $\widecheck{R}_{\g_n}^\lambda\amod$ that take projectives to projectives 
and 
define a categorical $\g_n$-action on $\widecheck{R}_{\g_n}^\lambda\amod$. 
This is constructed using the appropriate compositions of the various functors 
$\res^\eta_{\eta'}$, $\Pi^\eta_{\eta'}$, $F_j^\eta$, and $E_j^\eta$, and taking projective envelopes as in 
the previous subsection.  
We obtain the following theorem. 

\begin{thm}\label{thm:GTcat}
There is an isomorphism of $\g_n$ representations 
\begin{equation*}
K_0(\widecheck{R}_{\g_n}^\lambda)\xra{\ \ \cong\ \ } V_\lambda^{\g_n} 
\end{equation*} 
taking the projective ${}_{e(s)}L$ to the Gelfand-Tsetlin basis element $\ket{s}$. 
Moreover,
the isomorphism $K_0(\Pi^\lambda)\colon K_0(R_{\g_n}^\lambda)\to V_\lambda^{\g_n}$ 
of Theorem~\ref{thm:branchrules}  sends the projective 
${}_{e(s)}P$ to the Gelfand-Tsetlin basis element $\ket{s}$.
\end{thm}

\begin{proof} 
The first claim follows from Theorem~\ref{thm:branchrules} together with the fact that 
the ideal $\ker(\pi_s)$ is virtually nilpotent   
(this can be seen on a case-by-case check using the proofs of the Lemmas~\ref{lem:cycincB},~\ref{lem:cycincC}, 
and ~\ref{lem:cycincD} together with the diagrammatic description of the categorical projection 
given in Subsection 4.2 of~\cite{vaz}). 
The second claim is a consequence of the first one, since, for the submodule 
$K=\ker(\pi_s)\subset{}_{e(s)}P$,  
we have $K_0(K)=0$ and a s.e.s. 
$0\to K\to {}_{e(s)}P \to {}_{e(s)}L \to 0$.  
\end{proof}


\vspace*{1cm}


\subsection*{Acknowledgements}
The author thanks Daniel Tubbenhauer for 
numerous helpful comments and suggestions.


\end{document}